\documentclass[11pt]{amsart}
\usepackage{amsfonts}
\usepackage{fullpage}
\usepackage{amsmath,amsthm,bbm}
\usepackage{latexsym}
\usepackage{array}
\usepackage{dsfont}
\usepackage{amssymb}
\usepackage{graphics}
\usepackage{enumerate}
\usepackage[english]{babel}
\usepackage{pict2e}
\usepackage{xcolor}
\usepackage{pgf, tikz}

\usepackage{color}
\definecolor{marin}{rgb}   {0.,   0.3,   0.7} 
\definecolor{rouge}{rgb}   {0.8,   0.,   0.} 
\definecolor{sepia}{rgb}   {0.8,   0.5,   0.} 
\usepackage[colorlinks,citecolor=marin,linkcolor=rouge,
            bookmarksopen,
            bookmarksnumbered
           ]{hyperref}

\newtheorem{lemma}{Lemma}[section]

\newtheorem{theorem}[lemma]{Theorem}
\newtheorem{proposition}[lemma]{Proposition}
\newtheorem{corollary}[lemma]{Corollary}
\newtheorem{remark}[lemma]{Remark}
\newtheorem{example}[lemma]{Example}

\newtheorem{notation}[lemma]{Notation}
\newtheorem{definition}[lemma]{Definition}
\newtheorem{conclusion}[lemma]{Conclusion}
\numberwithin{equation}{section}
\newcommand{\QED}{\mbox{}\hfill \raisebox{-0.2pt}{\rule{5.6pt}{6pt}\rule{0pt}{0pt}} 
          \medskip\par}             

\newenvironment{Proofof}[1]{\noindent
    \parindent=0pt\abovedisplayskip = 0.5\abovedisplayskip
    \belowdisplayskip=\abovedisplayskip{\bfseries Proof of #1. }}{\QED}

\newcommand{\eps}{\varepsilon}

\newcommand{\dd}{\mathrm{d}}

\newcommand{\bE}{\boldsymbol{E}}
\newcommand{\bK}{\boldsymbol{K}}

\newcommand{\R}{\mathbbm{R}}

\newcommand{\Cc}{\mathcal{C}}
\newcommand{\Sc}{\mathcal{S}}
\newcommand{\C}{\mathbb{C}}

\newcommand{\T}{\mathbbm{T}}

\newcommand{\Z}{\mathbbm{Z}}

\newcommand{\Kc}{\mathcal{K}}

\newcommand{\am}{{\mathrm{am}}}
\newcommand{\sn}{\mathrm{sn}}
\newcommand{\cn}{\mathrm{cn}}
\newcommand{\dn}{\mathrm{dn}}
\newcommand{\rX}{\mathrm{X}}

\newcommand{\Norm}[2]{\|#1\|\left.\vphantom{T_{j_0}^0}\!\!\right._{#2}}

\author[E. Faou]{Erwan Faou}
\address{INRIA-Rennes Bretagne Atlantique and IRMAR (UMR 6625) Universit\'e de Rennes I} 
\email{Erwan.Faou@inria.fr}

\author[R. Horsin]{Romain Horsin}
\address{INRIA-Rennes Bretagne Atlantique and IRMAR (UMR 6625) Universit\'e de Rennes I} 
\email{Romain.Horsin@inria.fr}

\author[F. Rousset]{ Fr\'ed\'eric Rousset }
\address{ Laboratoire de Math\'ematiques d'Orsay (UMR 8628) Universit\'e Paris-Saclay}
\email{frederic.rousset@universite-paris-saclay.fr}
\urladdr{}

\thanks{This work was partially supported by the ERC starting grant GEOPARDI No. 279389}

\title[On linear damping around inhomogeneous stationary states of the Vlasov-HMF model]
{On linear damping around inhomogeneous stationary states of the Vlasov-HMF model}

\begin{document}


\maketitle

\begin{center}
{\em Dedicated to the memory of Walter Craig}
\end{center}

\begin{abstract}
We study the dynamics of perturbations around an inhomogeneous stationary state of the Vlasov-HMF (Hamiltonian Mean-Field) model, satisfying a linearized stability criterion (Penrose criterion). We consider solutions of the linearized equation around the steady state, and prove the algebraic decay in time of the Fourier modes of their density. We prove moreover that these solutions exhibit a scattering behavior to a modified state, implying a linear Landau damping effect with an algebraic rate of damping.
\end{abstract}

\nocite{*}

\section{Introduction}

In this paper we consider the Vlasov-HMF  (Hamiltonian Mean-Field) model. It is an ideal toy model that keeps several features of more complex kinetic equations, such as the Vlasov-Poisson system. It is moreover rather easy to do numerical simulations and analytic calculations on it, and it has thus received much interests in the physics literature (see  \cite{Barre4,Barre3,Barre1,Barre0,Barre2,Caglioti2,Caglioti3,Chavanis1,Chavanis2,Chavanis5,Chavanis4,Chavanis3}). This model exhibits also analogies with the Kuramoto model of coupled oscillators in its continuous limit \cite{Caglioti,Dietert,Gerard-Varet}. A long time analysis of the Vlasov-HMF model around spatially homogeneous stationary states has been performed in \cite{FR}, where a nonlinear Landau damping result is proved in Sobolev regularity.
In this paper, we consider the case of inhomogeneous steady states and study the long time behavior of the linearized equation. 

The Vlasov-HMF equation, with an attractive potential, reads
\begin{equation}
\label{V-HMF} 
\left|
\begin{split}
&\partial_{t}f(t,x,v)+\left\{f,H[f]\right\}(t,x,v)=0,\\
&H[f](t,x,v)=\frac{v^{2}}{2}-\phi[f](t,x), \\
&\phi[f](t,x)=\int_{\mathbbm{T}\times\mathbbm{R}}\cos(x-y)f(t,y,v)\dd y\dd v,
\end{split}
\right.
\end{equation}
with $(t,x,v)\in \mathbbm{R}\times\mathbbm{T}\times\mathbbm{R},$ where $\mathbbm{T}=\mathbbm{R}/\mathbbm{Z},$ and where 
\begin{equation}
\label{Poisson}
\left\{f,g\right\}=\partial_{x}f\partial_{v}g-\partial_{v}f\partial_{x}g
\end{equation}
is the Poisson bracket. It is rather easy to prove that this equation is globally well-posed, in Sobolev regularity for instance, using standard tools for transport equations associated with a divergence-free vector field.\\
The potential can be also  expressed as the following trigonometric polynomial
\begin{equation}
\label{potentialdef}
\phi[f](x)= \Cc[f]\cos (x) +  \Sc[f] \sin(x), 
\end{equation} 
with 
$$\Cc[f] = \int_{\T \times \R} \cos(y) f(y,v) \dd y \dd v\quad \mbox{and} \quad \Sc[f] = \int_{\T \times \R} \sin(y) f(y,v) \dd y \dd v,$$ 
where we use the  normalized Lebesgue measure on the torus\footnote{the Lebesgue measure on $[-\pi,\pi]$ divided by the length of the torus $2\pi$}. 

The previous equation possesses stationary solutions of the form 
\begin{equation}
\label{cimaros}
\eta(x,v)=G\left(h_{0}(x,v)\right), \quad h_{0}(x,v)=\frac{v^{2}}{2}-M_{0}\cos(x), \quad M_{0}>0.
\end{equation}
for some function  $G:\R \to \R$. The constant $M_0$ (called the magnetization) has to fulfill the condition
\begin{equation}
\label{eqM0}
M_0 = \Cc\left[G\left(\frac{v^2}{2} - M_0 \cos(x)\right) \right]. 
\end{equation}
Up to translation $x \mapsto x + x_0$, these are essentially the only stationary solutions, see Section \ref{sectionromain} where examples of couple $(M_0,G)$ satisfying the previous condition and the necessary stability condition ensuring damping effects are studied. 

For such a stationary states, if we seek solution of \eqref{V-HMF} under the form 
$f(t,x,v)= \eta(x,v) + r(t,x,v)$ with initial condition $r(0,x,v) = r^0(x,v)$, we obtain the equation
$$
\partial_{t}r(t,x,v)- \left\{\eta,\phi[r]\right\}(t,x,v)+\left\{r,H[\eta]\right\}(t,x,v)- \left\{r,\phi[r]\right\}(t,x,v)=0.
$$
{\em In this paper, we will retain the linear part of this equation, namely the linearized equation around $\eta$,} given by 
\begin{equation}
\label{V-HMF-lin}
\partial_{t}r(t,x,v)- \left\{\eta,\phi[r]\right\}(t,x,v)+\left\{r,h_0\right\}(t,x,v)=0. 
\end{equation}
The goal of this paper is the analysis of the long time behavior of this equation. 

In the homogeneous case $M_0 = 0$ where the steady steady states $\eta$ depends only on the velocity variable $v$, the situation both for the linear and nonlinear equation has been widely studied for general Vlasov equations. In this case, $h_0(x,v) = \frac{v^2}{2}$  the flow of the Hamiltonian $h_0$ is trivially calculated: without the potential term $\phi[r]$ in \eqref{V-HMF-lin}, the solution is given explicitly by $r(t,x,v) = r^0(x - tv,v)$ and gives rise to damping effect (see Landau \cite{Landau}) implying   a weak convergence of $r(t,x,v)$ towards the average of $r^0$ with respect to $x$. It then turns out that under a stability condition on $\eta$ called the Penrose condition, then the flow of the full linear equation \eqref{V-HMF-lin} behaves like the transport part for large times. This is well expressed as a scattering result where $g(t,x,v):= r(t,x + tv,v)$ is shown to converge for large times towards a smooth function $g_\infty(x,v)$ depending on $r^0(x,v)$. With this result in hand, the weak limit of $r(t,x,v)$ can be identified when $t \to +\infty$. 
 The scattering convergence rate depends on the regularity of the solution and is expected to be typically exponential for Gevrey or analytic functions and polynomial in time for finite Sobolev regularity. 
This linear scattering under a Penrose condition is the starting point of the nonlinear results of \cite{MV,BMM, GNR} showing that this scattering behavior persists in nonlinear equations. These result were proven for Gevrey initial data and general Vlasov equation including in particular the Vlasov-Poisson system. For the Vlasov-HMF a similar result can be proved under Sobolev regularity in the homogeneous case, see \cite{FR}. The question of Landau damping in Sobolev regularity for the Vlasov-Poisson system has been recently addressed, for instance in \cite{Bed1,Tristani}, where Landau damping results are proved in a weakly collisional regime, or in \cite{BMM2, HNR1, HNR2, BMM3} in the case of unconfined systems.

In the non-homogeneous case $M_0 > 0$ studied in this paper, the situation has been recently investigated (in particular in the physics literature see \cite{Barre3,Barre1} and in \cite{Dep19}). In this paper, we propose to follow the same strategy as in the homogeneous case where the free flow $x \mapsto x + tv$ is replaced by the flow  $\psi_t(x,v)$  of the Hamiltonian $h_0,$ associated with the ordinary differential equation 
\begin{equation}
\label{eq:pendulum}
\left\{
\begin{array}{rcll}
\dot x &=& \partial_v h_0 (x,v) &= v\\[2ex]
\dot v &=& - \partial_x h_0(x,v) &= - M_0 \sin(x),
\end{array}
\right.
\end{equation}
which is the classical dynamical system for the motion of a Pendulum. The flow $\psi_t$ is globally well defined and symplectic. In particular it preserves the Poisson bracket
$$
\forall\, t \in \R,\quad \{ f,g\} \circ \psi_t (x,v) = \{ f \circ \psi_t ,g\circ \psi_t\} (x,v). 
$$
Note that for a given function $f(x,v)$, the function 
$(t,x,v) \mapsto f(t,x,v) = f(\psi_t(x,v))$ solves the equation 
$\partial_t f = \{ f, h_0\}$ and $r \circ \psi_{-t}$ is the solution of the free flow in \eqref{V-HMF-lin}.

In this paper, we shall prove that under appropriate assumptions on $\eta$ (of Penrose type) and $r^{0}$ (a natural orthogonality condition), the solution of the linear equation \eqref{V-HMF-lin} also exhibit a scattering behavior: 
$g= r \circ \psi_{t}$ converges towards a function $g_{\infty}$, implying that the coefficients 
$\Sc[r(t)]$ and $\Cc[r(t)]$ decay in time, with algebraic rates of damping which depends on the regularity of the initial data and on the behavior of the function in the vicinity of the origin $(x,v) = (0,0)$ (the center of the ``eye" of the pendulum). 

The main ingredient of the proof is the use of action-angle variables $(\theta,a)$ for the integrable flow \eqref{eq:pendulum} and for which the flow $\psi_t$ can be calculated 
$\psi_t(\theta,a) = \theta + \omega(a) t$ for some frequency function $\omega$. We will use this ``explicit" formula (up to the knowledge of elliptic functions) to prove that $\psi_t$ gives rise to damping effect and that for smooth functions $\varphi$, the function of the form $\varphi\circ \psi_t$ has a weak limit that can be nicely expressed in terms of action-angle variables, and with a convergence rate in time $t$ depending on the ``flatness" of $\varphi$ and of the observable near the origin. This last particularity reflects the singularity of the action-angle change of variable. The full statement of this result is given in Theorem \ref{theodecay} which is proven in Section \ref{actionangle} and requires the use of precise asymptotics of Jacobi elliptic functions. This makes the proof seemingly technical, but the arguments are in fact simple for a reader familiar with this literature (we make a crucial use of many formulas in the book \cite{Elliptic1}). 

With this result in hand, our main results give the decay of the functions $\Sc(t)$ and $\Cc(t)$ and the convergence of $g = r \circ \varphi_t$, (and weak convergence of $r$, see Theorem \ref{theomagnetiz} and Corollary \ref{weakconv}) under an orghogonality assumption for $r^0$ well expressed in action-angle variable, and a Penrose condition on $\eta$, \eqref{penrose}. 

We then conclude by showing the existence of stationary states $\eta$ {\em i.e.} of couple $(M_0,G)$ satisfying \eqref{eqM0}) fulfilling the stability condition and relate it with more classical stability condition from the physics literature \cite{Barre3} that was also used in \cite{Orbital} conditioning  the nonlinear orbital stability of the inhomogeneous steady states of Vlasov-HMF. Strikingly enough, the key argument relies on explicit formulae in the action-angle change of variable and the direct verification that some terms do not vanish from known Fourier expansions of elliptic functions that can be found in \cite{Elliptic1}.
Note finally that the extension of our scattering result to the nonlinear case remains for the moment an open question. 

\section{Statements of the main results}

%We shall essentially prove that the situation remains similar in the inhomogeneous case, due to the special form of the steady states of Vlasov-HMF. Indeed, as these states are functions of the hamiltonian $h_{0}$ associated with the Pendulum system, we shall capture the phase-mixing effect by considering the angle-action variables of the Pendulum, in which $\psi_{t}(x,v)$ is integrable.\\
%As discussed in \cite{MV}, linear Landau damping is moreover typically proven by assuming that the steady state satisfies a linearized stability condition (the Penrose criterion), which may be seen as a a sufficient condition to solve the Volterra equations \eqref{Volt} by means of Fourier or Laplace transforms, in the style of L. Landau's original work \cite{Landau}. We shall make a similar assumption here, which we shall relate to a more classical stability condition from the Physics literature \cite{Barre3} that was also used in \cite{Orbital} to prove the nonlinear orbital stability of the inhomogeneous steady states of Vlasov-HMF.
%

We now fix the notations and give the main results. The first part shows the dispersive effect of the flow of the Pendulum, and the second part gives the main application for the long time behavior of the linear equation \eqref{V-HMF-lin}. 

\subsection{Damping in action-angle variables}

As a one-dimensional Hamiltonian system, the  system associated with the Hamiltonian 
$h_0(x,v)$ is integrable. We will need relatively precise informations about the corresponding action-angle change of variable. Let us split the space into three charts $U_+$, $U_-$ and $U_\circ$ as follows: 
\begin{equation}
\begin{split}
&U_{+} = \{ (x,v) \in \T \times \R \, |\, v > 0 \quad \mbox{and}\quad 
h_0(x,v) > M_0\}, \\
&U_{-} = \{ (x,v) \in  \T \times \R \, | \, v < 0 \quad \mbox{and}\quad 
h_0(x,v) > M_0\}, \quad\mbox{and}\quad \\
&U_{\circ} = \{ (x,v) \in  \T \times \R \, | \, 
h_0(x,v) < M_0\}. 
\end{split}
\end{equation}
We have that $h_0(x,v) \geq -M_0$ and the center of the ``eye" $U_\circ$ corresponds to the point $(x,v) = (0,0)$ which minimizes $h_0.$ The set
$$\left\{ (x,v)\in \mathbbm{T}\times \mathbbm{R} \quad | \quad h_{0}(x,v)=M_{0}\right\}$$
will usually be called the ``separatix". Let us first recall the following Theorem: 

\begin{theorem}
\label{theoaction-angle}
Setting $h(x,v) = \frac{v^2}{2} - M_0 \cos(x)
$, then for $* \in \{\pm, \circ\}$, there exists a symplectic change of variable $(x,v) \mapsto (\psi,h)$ from $U_{*}$ to the set 
$$V_* := \{(\psi,h) \in \R^2 \,  |\, h \in I_*, \, \psi \in (-r_*(h),r_*(h))\}, $$
where $r_*(h)$ is a positive function, $I_{\pm} = (M_0,+\infty)$ and $I_\circ = (-M_0,M_0)$
such that the flow of the pendulum in the variable $(\psi,h)$ is $h(t) = h(0)$ and $\psi(t) = t + \psi(0).$\\ 
Moreover, there exists a symplectic change of variables $(\psi,h) \mapsto (\theta,a)$ from $V_*$ to 
$$W_* = \{(\theta,a) \in \R^2 \, | \, \theta \in (-\pi,\pi),\, a \in J_*\}= \T \times J_*, $$
with $J_{\pm} = (\frac{4}{\pi}\sqrt{M_0},+\infty)$ and $J_\circ = (0,\frac{8}{\pi}\sqrt{M_0})$ 
such that 
$$\theta(\psi,h) = \omega_*(h) \psi, \quad \mbox{and} \quad \partial_h a(h) = \frac{1}{\omega_*(h)} = \frac{\pi}{r_*(h)}, $$
so that the flow of the pendulum in the variables $(\theta,a)$ in $W_*$ is $a(t) = a(0)$ and $\theta(t) = t \omega_{*}(a(0)) + \theta(0).$
\end{theorem}

This Theorem is explicit in the sense that the changes of variables express in terms of Jacobi elliptic functions. As $\theta$ is a variable in a fixed torus, Fourier series in variable $\theta$ are well defined on each set $W_*$ corresponding to $U_*$. For a given function $f(x,v)$ we can define the restriction $f^*$ of $f$ to the set $U_*,$ and the Fourier coefficients 
\begin{equation}
\label{coefou}f^*_\ell(a) = \frac{1}{2\pi} \int_{-\pi}^{\pi} f^*(  x(\theta,a), v(\theta,a))  e^{-i \ell \theta} \dd \theta, \qquad \ell \in \Z, \quad a\in J_*
\end{equation}
where $x(\theta,a)$ and $v(\theta,a)$ are given by the change of variable on $U_*$.  
Note that for given functions $f$ and $\varphi$, we have the decomposition 
\begin{equation}
\label{oslo}
\int_{U_{*}} f(x,v) \varphi(x,v) \dd x \dd v = \sum_{\ell\in \Z} \int_{J_*} f^*_\ell (a) \varphi^*_{-\ell} (a) \dd a, 
\end{equation}
for all $*\in\{\circ,\pm\}.$ Finally, let us notice that the Jacobian of the change of variable $h \mapsto a(h)$ is $\partial_h a(h) = \frac{1}{\omega^*(h)}$, and we have in particular 
\begin{equation}
\label{trondheim}
 \int_{J_*} f^*_\ell (a) \varphi^*_{-\ell} (a) \dd a =  \int_{I_*} f^*_\ell (a(h)) \varphi^*_{-\ell} (a(h)) \frac{1}{\omega_*(h)}\dd h. 
\end{equation}
We will usually write $f^*_\ell(h)$ for the quantity  $f^*_\ell(a(h))$, and several times consider functions $f$ as depending on $(x,v)$, $(\theta,h)$ and $(\theta,a)$ by keeping the same notation. For example a stationnary state $\eta$ depends only on $h$ and hence on $a$ and will be written $\eta = G(h)$ or $\eta = G(a)$. 

In fact the singularities of the relevant functions in action-angle variables are better expressed in variables $(\theta,h)$, which are not symplectic, but on which integrals and flow of the system are easy to calculate. 
Moreover, in this case, 
$$f^*_0(a) =  \frac{1}{2\pi} \int_{-\pi}^{\pi} f^*( x(\theta,h), v(\theta,h)) \dd \theta$$
can be seen as an average of $f$ on the isocurve $\{ (x,v) \, | \,  h_0(x,v) = h\}$, while $\psi$ is the arclength on this curve, the jacobian $\frac{1}{\omega_{*}(h)}$ appearing in the standard co-area formula, which is another way to see \eqref{oslo}-\eqref{trondheim}. \\
The notations $C_{\ell}^{*}(a)$ and $S_{\ell}^{*}(a)$ will be used for the Fourier coefficients of the functions
$$\theta \mapsto \cos(x(\theta,a)) \quad \mbox{and} \quad \theta \mapsto \sin(x(\theta,a)),$$
respectively, and both restricted to $U_{*}.$ These coefficients can be calculated explicitly using elliptic functions (see Propositions \ref{prop74} and \ref{prop710}), and we shall write
\begin{equation}
\label{CS}
\cos(x(\theta,a)) = \sum_{\ell \in \Z} C_\ell^*(a) e^{i\ell \theta} \quad \mbox{and} \quad \sin(x(\theta,a)) = \sum_{\ell \in \Z} S_\ell^*(a) e^{i\ell \theta},
\end{equation}
for $(\theta,a)\in J_{*}\times (-\pi,\pi).$\\
Before stating our first result, let us fix some notations. We use the classical notation $\langle v \rangle = (1+|v|^{2})^{1/2}$, for any $v \in \R^d$, and for a two-dimensional integer $\alpha=(\alpha_{1},\alpha_{2})\in\mathbbm{N}^{2},$ we set $|\alpha|=\alpha_{1}+\alpha_{2}$. 
We shall also write $\partial_{x,v}^{\alpha}$ for the operator acting on functions $f:\T\times \R \to \mathbbm{C}$ by the formula
$$\partial_{x,v}^{\alpha}f(x,v)=\partial_{x}^{\alpha_{1}} \partial_{v}^{\alpha_{2}} f(x,v).$$
In section \ref{actionangle} we prove the following result: 
\begin{theorem}[Dispersive effect of the pendulum flow]
\label{theodecay}
Consider $f(x,v)$ and $\varphi(x,v)$ two functions such that 
$$\max_{|\alpha |  \leq m} \Norm{ \langle v \rangle^\mu\partial_{x,v}^\alpha  f(x,v)}{L^{\infty} }\leq C_{m,\mu}
\quad\mbox{and}\quad
\max_{|\alpha |  \leq M} \Norm{\partial_{x,v}^\alpha  \varphi(x,v)}{L^{\infty}}\leq C_{M},$$
for some $m$, $M$ and $\mu > 2$. 
Let $p$ and $q$ be defined by
\begin{equation}
\notag
\begin{split}
&p= \max \{n \geq 1  ,  \, \partial^\alpha_{x,v} f(0,0) = 0, \, \forall \alpha, \, 1 \leq |\alpha | \leq n\},\\
&q= \max \{n \geq 1  ,  \, \partial^\alpha_{x,v} \varphi(0,0) = 0, \, \forall \alpha, \, 1 \leq |\alpha | \leq n\}, 
\end{split}
\end{equation}
with the convention that these number are $0$ is the corresponding sets are empty. 
Then, if
$$m \geq  5 + p  + { p+ q \over 2} \quad \mbox{and} \quad M\geq \max\left(7 + q +  { p+ q \over 2 },  m+2\right),$$
there exists $C>0$ such that for all $t \geq 0,$ we have
$$\left|\int_{\T \times \R} f(x,v) \varphi(\psi_t(x,v)) \dd x \dd v - \sum_{* \in \{ \pm,\circ\}} \int_{J_*} f_0^*(a) \varphi_0^*(a)  \dd a \right|  \leq {\frac{C}{\langle t \rangle^{ \frac{p+q}{2} + 2 } }}.$$
\end{theorem}

Let us explain this Theorem as follows: the starting point of the proof is the Fourier expansion \eqref{oslo}, which yields
\begin{eqnarray*}
\int_{\mathbbm{T}\times \mathbbm{R}} f(x,v) \varphi(\psi_t(x,v)) \dd x \dd v  &=& \sum_{* \in \{ \pm,\circ\}} \sum_{\ell\in \Z} \int_{J_*}f_\ell^*(a) \varphi_{-\ell}^*(a)  e^{i t \ell \omega_{*}(a) } \dd a \\
&=&\sum_{* \in \{ \pm,\circ\}} \sum_{\ell\in \Z} \int_{I_*}f_\ell^*(h) \overline{\varphi_{\ell}^*(h)}  e^{i t \ell \omega_{*}(h) } \frac{1}{\omega_{*}(h)} \dd h. 
\end{eqnarray*}
Now we can use a stationary phase argument by integrating with respect to $h$ to gain a decay with respect to $t$. Typically, this kind of analysis depends on the possible cancellation of $\partial_h \omega_{*}(h)$. In our case, the situation seems to be very favourable, as $\partial_h \omega_{*}(h)$ never approaches zero, as shown in Section \ref{actionangle}. The stationary phase argument also relies on cancellations of $f_{\ell}*$ and $\varphi_{\ell}^{*}$ at the boundary points, and there the problems come from the singularities of the action-angle variables.\\
We can distinguish two zones, starting with the separatix $h \sim M_0$. In this case, the action-angle variables induce logarithmic singularities. Essentially it means that the Fourier coefficients $f^*_\ell,\varphi_{\ell}^{*}$ involve logarithmic singularities near $h = M_0$. However, near this point, $\omega_{*}(h)$ also exhibits a logarithmic singularity, and it can be shown that $\partial_h \omega_{*}(h)$ goes to infinity fast enough to ensure a decay in time which is essentially driven by the regularity of $f$ and $\varphi$. So the problems are not at the separatix.\\ 
Near the point $h = -M_0$, the situation is more delicate: in this zone, the pendulum Hamiltonian is essentially a perturbation of the Harmonic oscillator, for which no damping is expected ($\omega_{*}$ being constant). However, we can prove that $\partial_h \omega_{*}(h)$ does not vanish near this point. But this is not enough: indeed the action-angle variable of the harmonic oscillator involves algebraic singularity of order $\sqrt{h + M_0}$. This explains why the rate of decay of the integral with respect to the time is mainly driven by the behavior of $f$ and $\varphi$ near $(0,0)$ which corresponds of a local behavior of $f_\ell^\circ(h) \varphi_{-\ell}^\circ(h)$ in $(h + M_0)^{\frac{p+q}{2}},$ yielding the main contribution for the decay in the previous Theorem.\\
Theorem \ref{theodecay} will be a straightforward consequence of Propositions \ref{decaydehors} and \ref{decaydansloeil}, proven in section \ref{actionangle}.

\subsection{Linear damping}

As explained above, our main result is expressed as a scattering result with the strong convergence of $g = r \circ \psi_t$ and by using the previous Theorem, the weak convergence of $r$. The next proposition gives the equation satisfied by $g$ and fixes some notations used later. This Proposition will be proved at the beginning of section 3.

\begin{proposition}
\label{L-V}
Let $r(t,x,v)$ be the solution of the linearized equation \eqref{V-HMF-lin}. 
Then the function 
\begin{equation}
\label{gdef}
g(t,x,v) = r(t,\psi_{t}(x,v)) = r \circ \psi_t(x,v)
\end{equation}
satisfies the equation
\begin{equation}
\label{gequation}
\partial_t g = \Cc(t) \{ \eta , \cos ( \rX \circ \psi_{t})\} + \Sc(t) \{\eta, \sin ( \rX \circ \psi_{t})\}. 
\end{equation}
where $\rX:\T\times \R \to \T$ denotes the projection $\rX(x,v)=x,$ and where 
\begin{equation}
\label{Cdef0}
\Cc(t) = \Cc[r(t)]=\Cc[g\circ \psi_{-t}] = \int_{\T \times \R} \cos(\rX (y,w)) g(t,\psi_{-t}(y,w)) \dd y \dd w
\end{equation}
and
\begin{equation}
\label{Sdef0}
\Sc(t) = \Sc[r(t)]=\Sc[g\circ \psi_{-t}] = \int_{\T \times \R} \sin(\rX (y,w)) g(t,\psi_{-t}(y,w)) \dd y \dd w.
\end{equation}
Moreover, the coefficients $\Cc(t)$ and $\Sc(t)$ satisfy the following Volterra integral equations
\begin{equation}
\label{Volt}
\Cc(t)=F_{\Cc}(t) + \int_{0}^{t}\Cc(s) K_{\Cc}(t-s) \dd s \quad \mbox{and} \quad \Sc(t)=F_{\Sc}(t) + \int_{0}^{t}\Sc(s) K_{\Sc}(t-s) \dd s,
\end{equation}
with
\begin{equation}
\label{eq:F}
\begin{array}{rcl}
F_{\Cc}(t) &=& \displaystyle  \int_{\T \times \R} \cos(\rX \circ \psi_{t}(y,w)) r^{0}(y,w) \dd y \dd w ,\\[2ex]
 F_{\Sc}(t) &=& \displaystyle \int_{\T \times \R} \sin(\rX \circ \psi_{t}(y,w)) r^{0}(y,w) \dd y \dd w ,\\[2ex]
K_{\Cc}(t)&=&\displaystyle\mathds{1}_{\{t\geq 0\}}\int_{\T \times \R } \{ \eta , \cos ( \rX )\}\cos(\rX \circ \psi_t) , \quad\mbox{and}\\ [2ex]
K_{\Sc}(t)&=&-\displaystyle\mathds{1}_{\{t\geq 0\}}\int_{\T \times \R } \{ \eta , \sin ( \rX )\}\sin(\rX \circ \psi_t) .
\end{array}
\end{equation}
\end{proposition}

%
%Note that the coefficients $\Cc(t)$ and $\Sc(t)$ exist globally in time provided that $r^{0}(x,v)$ is smooth enough, by standard well-posedness results for \eqref{V-HMF-lin}. We should point that they satisfy a closed system of Volterra equations that will play the same crucial part in the dynamics than the Fourier modes of the density $\int_{v}r(t,x,v)\dd v$ in the case of Landau damping around spatially homogeneous steady states (see \cite{BMM,FR,MV}).
%
%%
%

For a function $F(t)$, we define its Fourier transform by
$$\hat F(\xi) = \int_{\R} F(t) e^{- i t \xi} \dd t.$$

\begin{theorem}[Linear damping]
\label{theomagnetiz}
Let $\eta(x,v) = G(h_0(x,v))$ with $G$ a decreasing function that satisfies the assumption
\begin{equation}
\label{regG}\max_{n\leq 10}\left\| \langle y \rangle^{\mu}  G^{(n)}(y) \right\|_{L^\infty(\R)}\leq C_{\mu},
\end{equation}
with $\mu >2,$ and assume that there exists $\kappa>0$ such that 
\begin{equation}
\label{penrose}
\min_{\mathrm{Im} (\xi) \leq 0} | 1 - \hat K_{\mathcal{C}}(\xi) | \geq \kappa\quad \mbox{and}\quad 
\min_{\mathrm{Im} (\xi) \leq 0} | 1 - \hat K_{\mathcal{S}}(\xi) | \geq \kappa. 
\end{equation}
Let us assume that the initial perturbation $r^{0}$ satisfies 
$$\max_{|\alpha |  \leq m} \Norm{ \langle v \rangle^\nu\partial_{x,v}^\alpha  r^{0}(x,v)}{L^{\infty} }\leq C_{m,\nu},$$
for some $\nu>2,$ and where
$$m\geq 5+\frac{3 p}{2},$$
with
$$p=\max\left\{k\geq 1, \quad \partial_{x,v}^{\alpha} r^{0}(0,0)=0, \quad \forall 1\, \leq |\alpha|\leq k\right\}.$$
Then, if $r^{0}$ satisfies the orthogonality condition 
\begin{equation}
\label{orthocond}
\sum_{* \in \{\pm,\circ\}} \int_{J_*} C_0^*(a) (r^{0})_0^*(a)  \dd a  = 0, 
\end{equation}
there exists $C>0$ such that for all $t \geq 0$
$$ |\mathcal{C}(t)| \leq \frac{C}{\langle t \rangle^{ \max(3,\frac{p+5}{2}) }} \quad \mbox{and} \quad |\mathcal{S}(t)| \leq \frac{C}{\langle t \rangle^{ 2}}.$$
%with
%\begin{itemize}
%\item[$\bullet$] $\alpha = 5/2 $ if $p\geq1$ or $\int_{\R\times \T} r^{0}(x,v) \dd x \dd v = 0,$
%\item[$\bullet$] $\alpha = 3$ if both conditions are satisfied,
%\item[$\bullet$] $\alpha =2$ if none of them is.
%\end{itemize}
\end{theorem}

%{\bf a)} The damping rates of the two coefficients $\Cc(t)$ and $\Sc(t)$ are algebraic, and essentially independent of the assumptions on the regularity of $r^{0}.$ This may be explained as follows: by using a Paley-Wiener Theorem, we will prove that, under the assumptions on $\hat{K}_{\Cc}$ and $\hat{K}_{\Sc},$ the quantities $\Cc(t)$ and $\Sc(t)$ inherit the decay in time of the source terms $F_{\Cc}(t)$ and $F_{\Sc}(t)$ of the two Volterra equations that they satisfy (see \eqref{Volt}). Now these sources terms fall under the scope of Theorem \ref{theodecay}, and thus the damping rates are mainly driven by the behavior of the functions $\cos( X) r^{0}$ and $\sin(X) r^{0}$ in the vicinity $(x,v)\sim (0,0),$ and both of them have a zero of finite order there, independently from the regularity of $r^{0}$.\\
%This is of course a major difference with the existing results on Landau damping around spatially homogeneous stationary states, studied for instance in \cite{BMM,FR,MV}. In those papers, the role of $\Cc(t)$ and $\Sc(t)$ in the dynamic is played by the complex Fourier modes of the density $\int _{v}f(t,x,v)\dd v,$ which satisfy as well Volterra integral equations. In this case however, the decay in time of the source terms is completely driven by the regularity of $r^{0}.$\\

 We shall call assumption \eqref{penrose} the Penrose criterion, by analogy with the stability conditions of the same name in the homogeneous case. 
 
Let us remark that the orthogonality condition \eqref{orthocond} is propagated by the flow of the linear equation
\eqref{gequation} and therefore natural to impose. 
Indeed by using the action-angle variables given by Theorem \ref{theoaction-angle}, we have that $\eta = G(h)$ is in fact a function of $h$ and hence of $a$ only. Hence the equation \eqref{gequation} in symplectic  variables $(\theta,a)$ can be written 
$$ \partial_{t} g = \mathcal{C}(t) \{ G(a), \cos( x(\theta+ t\omega(a), a)) \} + \mathcal{S}(t)   \{ G(a), \sin( x(\theta+ t\omega(a), a)) \}$$
Since $G(a)$ depends only on $a$, we get from the above equation that
$$
\partial_{t}  g^*_0(t, a) = \partial_{t} \int_{\T} g^*(t, \theta, a) \, \dd\theta= 0, \quad a \in I_*, \quad * \in \{\pm,\circ\}. 
$$
As $g_0^*(t,a)  = r^*_0(t,a)$, this shows that 
 that the orthogonality condition \eqref{orthocond} is propagated along the flow of \eqref{V-HMF-lin}. 
\medskip

As a corollary of Theorem \ref{theomagnetiz}, we get a scattering result for the solution $g$ of \eqref{gequation} and the weak convergence of $r(t,x,v)= g(t, \psi_{-t}(x,v))$ the solution of \eqref{V-HMF-lin} towards an asymptotic state $r_{\infty}(x,v)$  that depends only on $h_{0}(x,v).$

\begin{corollary}
\label{weakconv}
Under the assumptions of Theorem \ref{theomagnetiz} with $p=0$, we obtain that:
\begin{itemize}
\item There exists $g_{\infty}(x,v)$ and a constant $C$ such that  when $t \rightarrow + \infty$, we have
\begin{equation}
\label{scat}
\| g(t) - g_{\infty}\|_{L^1_{x,v}} \leq {C \over \langle t\rangle}.
\end{equation}
\item There exists  $r_{\infty}(x,v)$ that depends only on $h$, that is to say $ r_{\infty}(x,v) = r_{\infty}^*(h)$ for $(x,v) \in U_*$ and $* \in \{ \pm,\circ\}$,  
such that for every test function $\phi$, we have that
$$ \int_{\T \times \R} r(t,x,v) \phi(x,v)\, \dd x\dd v \rightarrow_{t \rightarrow + \infty} \int_{\T \times \R} r_{\infty}(x,v) \phi(x,v)\, \dd x \dd v.$$
\end{itemize}
\end{corollary}

%Let us make the following additional comments:
%
%\medskip
%
%{\bf a)} In this special case of the linearized dynamics \eqref{V-HMF-lin}, we can express $r_{\infty}$ in terms of the initial data $r^{0}$, 
%and obtain that
%$$R_{\infty}(h) = \int_{-\pi}^{\pi} r^{0}(x^*(\theta, h), v^*(\theta, h)) d \theta,  \quad h \in I_{*}, \, *\in \{\pm, \circ\}.$$
%It also shows that
%$$f(\theta,h) = \eta (h) + r(t,x^*(\theta, h), v^*(\theta, h)), \quad h \in I_{*}, \, *\in \{\pm, \circ\},$$
%converges weakly to its initial average with respect to $\theta.$\\
%This is a behavior that we should relate to linear Landau damping around spatially homogeneous stationary states, since in this case the function
%$$f(t,x,v) =\eta(v)+r(t,x,v)$$
%converges to its initial average with respect to $x$ (see \cite{MV} for a complete discussion on linear Landau damping). 
%
%\medskip
%
%{\bf b)} The first estimate of Corollary \ref{weakconv} shows that $g$ exhibits a scattering behavior, which is also the case in both linear and nonlinear Landau damping around spatially homogeneous stationary states. The difference is that the present scattering estimate \eqref{scat} holds in $L^{1}$ norm, whereas in the homogeneous case scattering estimates may be proved in regular norm, such as Sobolev's (\cite{FR}) or Gevrey's (\cite{BMM}). The question of linear (and nonlinear) Landau damping around inhomogeneous stationary states in regular norms is left open.
    
\subsection{About the Penrose criterion}
    
Written in this form, the Penrose criterion \eqref{penrose} is difficult to check, but we can relate it to a more classical condition that was found in \cite{Orbital} or \cite{Barre3} to ensure orbital stability of inhomogeneous stationary states in the nonlinear equation.\\  
First we shall prove that the verification of Penrose criterion \eqref{penrose} at the frequency $\xi=0$ is sufficient, by proving the following Theorem.

\begin{theorem}
\label{theopenrose}
Let $\eta$ be a state defined by \eqref{cimaros}, and assume that $G$ satisfies the regularity assumption \eqref{regG}. Assume moreover that $G'<0$ and 
\begin{equation}
\label{penrose0}
1 - \hat K_{\mathcal{C}}(0)  > 0   \quad \mbox{and}\quad  1 - \hat K_{\mathcal{S}}(0) > 0.
\end{equation}
Then the Penrose criterion \eqref{penrose} holds true. 
\end{theorem}

Let us now define the following notion of stability (see also \cite{Barre3,Orbital}).

\begin{definition}
\label{phy}
A state $\eta(x,v) = G(h_0(x,v))$ defined by \eqref{cimaros} is said to be linearly stable if
\begin{equation}
\label{eqphy}
1+\int_{\R\times \T} G'(h_0(x,v)) \cos^{2}(x)\dd x\dd v - \sum_{*\in\{\pm,\circ\}} \int_{J_{*}} G'(h_0(a)) C_{0}^{*}(a)^{2} \dd a > 0.
\end{equation}
\end{definition}

For regular and decreasing profile $G$, we first show that this condition is equivalent to the previous one: 

\begin{proposition}
\label{penrosephy}
Let $\eta$ be a state defined by \eqref{cimaros}. Assume that $G$ satisfies the hypothesis \eqref{regG}, and that $G'<0.$ Then \eqref{penrose0} holds true if and only if $\eta$ is stable in the sense of Definition \ref{phy}.
\end{proposition}

Finally, we exhibit examples of stable stationary states given by Maxwell-Boltzmann distribution, under some condition on the coefficients of the Gaussian: 
\begin{proposition}
\label{stab-gauss1}
Let $\alpha > 0$ and $\beta >0$ such that $\alpha^2 \beta<\frac{2}{\pi}$, then there exists $M_0 > 0$ satisfying \eqref{eqM0} such that 
$$\eta(x,v)=\alpha e^{-\beta\left(\frac{v^{2}}{2}-M_{0}\cos(x)\right)},$$
is a stable stationary states in the sense of Definition \ref{phy}. 
%with $\alpha,\beta,M_{0} \in \R_{+}^{*},$ $\alpha\sqrt{\beta}<\frac{2}{\sqrt{2\pi}},$ and $M_0$ given by Proposition \ref{general}, and satisfying
%$$M_{0}=\alpha\int_{\T\times \R} e^{-\beta \left(\frac{v^{2}}{2}-M_{0}\cos(x)\right)} \cos(x) \dd x\dd v.$$\\
%Then $\eta$ is stable in the sense of definition \ref{phy}.
\end{proposition}

\subsection{Organization}

In section 3, we collect and prove some results concerning Volterra integral equations, and use them to prove the linear damping Theorem \ref{theomagnetiz} by assuming Theorem \eqref{theodecay} giving the dispersive effect of the flow of the Pendulum.
In section 4, we prove the scattering result corollary \ref{weakconv}. Section 5 is dedicated to the Penrose criterion, and we prove there Theorem \ref{theopenrose} and Proposition \ref{penrosephy}. In section \ref{sectionromain} we exhibit examples of inhomogeneous stationary states which are stable in the sense of definition \ref{phy} and prove Proposition \ref{stab-gauss1}. Finally, section 7 contains all the technical results that we shall need concerning angle-action variables, and we prove there the dispersion Theorem \ref{theodecay}.

\section{Proof of the linear damping Theorem \ref{theomagnetiz}}

We begin with the derivation of the Volterra equation \eqref{Volt}. The proof then consists in 
showing that the kernels $K_{\mathcal{C}}$ and $K_{\mathcal{S}}$ have sufficient decay in time, which, with the Penrose criterion and a Paley-Wiener argument will yield a control of the decay in time of $\Cc(t)$ and $\Sc(t)$ by the one of the source terms $F_{\Cc}(t)$ and $F_{\Sc}(t),$ and the latter will be guaranteed by Theorem \ref{theodecay}. 

In all the remainder of the paper, we will often use the notation $A \lesssim B$ to denote an inequality of the form $A \leq CB$ for some constant $C$ depending only on the assumptions made in the section of the proof but not on $A$ or $B$.

\begin{Proofof}{Proposition \ref{L-V}}
Let us first prove \eqref{gequation}. If $r$ solves \eqref{V-HMF-lin}, the function $g$ defined in \eqref{gdef} 
satisfies
\begin{eqnarray*}
\partial_t g (t,x,v) &=& \{ r, h_0\} (t,\psi_{t}(x,v)) + \left\{\eta,\phi[r]\right\}(t,\psi_{t}(x,v))-\left\{r,h_0\right\}(t,\psi_{t}(x,v))\\
&=& \{\eta,\phi[g \circ \psi_{-t}]\}(t,\psi_{t}(x,v)).
\end{eqnarray*}
Hence as $\eta$ is invariant by the flow $\psi_t$, $g$ solves 
$$\partial_t g(t,x,v) = \{ \eta, \phi[g\circ \psi_{-t}] \circ \psi_{t}\} (t,x,v). $$
Since $\psi_t$ preserves the volume, 
\begin{eqnarray*}
\phi[g\circ \psi_{-t}] \circ \psi_{t}(x,v) &=& \int_{\T \times \R} \cos(\rX \circ \psi_{t}(x,v) - y) g(t,\psi_{-t}(y,w)) \dd y \dd w\\
&=& \int_{\T \times \R} \cos(\rX \circ \psi_{t}(x,v) - \rX(y,w)) g(t,\psi_{-t}(y,w)) \dd y \dd w \\
&=& \cos(\rX \circ \psi_{t}(x,v)) \Cc(t) + \sin(\rX \circ \psi_{t}(x,v)) \Sc(t),
\end{eqnarray*}
with $\Cc(t)=\Cc[g\circ \psi_{-t}]=\Cc[r(t)]$ and $\Sc(t)=\Sc[g\circ \psi_{-t}]=\Sc[r(t)]$, 
which proves \eqref{gequation}.\\
We deduce that
$$g(t,x,v) = r^{0}(x,v) + \int_0^t \Cc(s) \{ \eta , \cos \left( \rX \circ \psi_{s}\right)\} + \Sc(s) \{\eta, \sin \left( \rX \circ \psi_{s}\right)\} \dd s. $$
Using this formula and the fact that $\psi_{t}$ preserves the Poisson bracket, we calculate that 
\begin{equation}
\notag
\begin{split}
 \Cc(t) =  \int_{\T \times \R} \cos(\rX(y,w)) g(t,\psi_{-t}(y,w)) \dd y \dd w &= \int_{\T \times \R} \cos(\rX(y,w)) r^{0}(\psi_{-t}(y,w)) \dd y \dd w  \\
 &+ \int_0^t \Cc(s) \int_{\T \times \R } \cos(X)\{ \eta , \cos ( \rX \circ \psi_{s-t})\} \dd s \\
 &+ \int_0^t\Sc(s)\int_{\T \times \R} \cos(X) \{\eta, \sin (X \circ \psi_{s-t})\} \dd s. 
\end{split}
\end{equation}
Note that the flow $\psi_t$ is reversible with respect to the transformation $\nu(x,v) = (x,-v),$
that is we have $\psi_t \circ \nu = - \nu \circ \psi_{-t}$. But as the Hamiltonian is even in $x$, the flow is also reversible with respect to $(x,v) \mapsto (-x,v)$. Hence the transformation $\mu(x,v) := (-x,-v)$ satisfies $\psi_t \circ \mu = \mu \circ \psi_t$, and this transformation preserves the Poisson bracket and is an isometry. 
Let us apply this to the last term in the previous equation. We thus have for any $\sigma \in \R$
\begin{multline*}
\int_{\T \times \R} \cos(\rX) \{\eta, \sin (\rX \circ \psi_{\sigma})\}= \int_{\T \times \R} \cos(\rX\circ \mu) \{\eta, \sin (\rX \circ \psi_{\sigma})\}\circ \mu  \\
= \int_{\T \times \R} \cos(\rX) \{\eta, \sin (\rX \circ \mu \circ \psi_{\sigma})\} =  - \int_{\T \times \R} \cos(\rX) \{\eta, \sin (\rX \circ \psi_{\sigma})\} = 0, 
\end{multline*}
as $\rX \circ \mu = -\rX$. For the same reason, we have 
$$\int_{\T \times \R} \sin(\rX) \{\eta, \cos (\rX \circ \psi_{\sigma})\} = 0. $$
Now using the identities $\eta\circ \nu = \eta$ and $\rX\circ \nu =\rX,$ and the evenness of the cosine function, we have
\begin{multline*}
\int_{\T \times \R } \cos(\rX)\{ \eta , \cos ( \rX \circ \psi_{s-t})\}=-\int_{\T \times \R } \cos(\rX\circ \nu)\{ \eta , \cos ( \rX \circ \psi_{s-t}\circ \nu)\}\\
=-\int_{\T \times \R } \cos(\rX)\{ \eta , \cos ( \rX \circ (-\nu) \circ \psi_{t-s})\}
=-\int_{\T \times \R } \cos(\rX)\{ \eta , \cos ( \rX \circ \psi_{t-s})\}.
\end{multline*}
Integrating by parts that last integral yields then
$$\int_{\T \times \R } \cos(\rX)\{ \eta , \cos ( \rX \circ \psi_{s-t})\}=\int_{\T \times \R } \cos(X\circ \psi_{t-s})\{ \eta , \cos ( \rX )\}.$$
Using the oddness of the sine function, we have by similar manipulations
$$\int_{\T \times \R } \sin(\rX)\{ \eta , \sin ( \rX \circ \psi_{s-t})\}=-\int_{\T \times \R } \sin(X\circ \psi_{t-s})\{ \eta , \sin ( \rX )\}.$$
This ends the proof.
\end{Proofof}

As a preliminary, we shall first use Theorem \ref{theodecay} in order to get the decay rates
of the kernels. We shall prove the following result.

\begin{proposition}
\label{propnoyaux}
Let $\eta(x,v) = G(h_0(x,v))$ with $G$ a decreasing function that satisfies the assumption \eqref{regG}
with $\mu >2$.   
Then there exist a constant $C$ such that 
\begin{equation}
\label{estKCKS}
|K_\mathcal{C}(t)| \leq \frac{C}{\langle t\rangle^3}\quad \mbox{and}\quad |K_\mathcal{S}(t) |\leq \frac{C}{\langle t\rangle^2}.
\end{equation}
\end{proposition}

\begin{proof}
%Let us first note that as
%$$\int_{\T\times \R} \left\{\eta,\cos(\rX)\right\} \dd x\dd v= -\int_{\T\times \R} v\sin(\rX) G'(h_{0}(x,v))  \dd x \dd v=0,$$
%by oddness, we actually may write that by using \eqref{eq:F}
%$$K_{\Cc}(t)=\mathds{1}_{\{t\geq 0\}} \int_{\mathbbm{T}\times \mathbbm{R}} \left\{\eta,\cos(\rX)\right\} (1-\cos(\rX\circ \psi_{t})).$$
In view of the expression \eqref{eq:F} or $K_{\Cc}(t)$ we apply Theorem \ref{theodecay} with the functions $f(x,v) = \left\{\eta,\cos(\rX)\right\}(x,v)$ and $\varphi = \cos(\rX(x,v))$ for which we have $p=1$ and $q=1$. As we have that for all $*\in \{\circ,\pm\}$,
\begin{equation}
\label{brahms3}f_{0}^{*}(h)=\frac{1}{2\pi}\int_{-\pi}^{\pi} f^{*}(x(h,\theta),v(h,\theta))\dd \theta=\frac{\omega_{*}(h)G'(h)}{2\pi}\int_{-\pi}^{\pi}\partial_{\theta}(\cos(x(h,\theta)))\dd \theta =0.
\end{equation}
Theorem \ref{theodecay} then yields 
$|K_{\Cc}(t)|\lesssim \frac{1}{\langle t \rangle^{3}}$.
Concerning $K_{\Sc}(t),$
it suffices to apply Theorem \ref{theodecay} with the functions $\left\{\eta,\sin(\rX)\right\}(x,v)$ and $\sin(\rX(x,v)).$ We have this time $p=q=0,$ and $S_{0}^{*}(h)=0$ for all $*\in \{\circ,\pm\}$ (see \eqref{coeffouriersin1} and \eqref{coeffouriersin2}). Hence the application of Theorem \ref{theodecay} yields
$|K_{\Sc}(t)|\lesssim \frac{1}{\langle t \rangle^{2}}$. 
\end{proof}

To study the coefficients $\Cc(t)$ and $\Sc(t),$ we shall use general results on Volterra integral equations written under the form 
\begin{equation}
\label{volterra1}
 y(t) = K * y (t) + F(t), \quad t \in \mathbb{R}
\end{equation}
where $K,  \, y, \, F $ vanish for $t \leq 0$.
Let us first recall the following Paley-Wiener result on Volterra integral equations (Theorem 4.1 of \cite{Volterra}, see also \cite{Dietert,Volterra2}).

\begin{lemma}[{Paley-Wiener}]
\label{lemvolterra}
Assume that  $  K \in L^1(\mathbb{R})$  is such that
$$ \min_{\mathrm{Im} (\xi) \leq 0} | 1 - \hat{K}(\xi) | \geq \kappa.$$
Then there exists a unique resolvent kernel $R \in L^1(\mathbb{R}_{+})$ which vanishes for $t \leq 0$ such that
\begin{equation}
\label{resolvantint}
R(t)= - K(t) + K*R(t).
\end{equation}
\end{lemma}
 
Note that using $R$, the solution of \eqref{volterra1} can be written as
\begin{equation}
\label{solres}
y(t)= F(t)- R*F(t).
\end{equation}
We shall then use the following corollary.
 
\begin{corollary}
\label{corvolterra}
Under the assumptions of Lemma \ref{lemvolterra}, the following holds:
\begin{itemize}
\item[i)] There exists $C>0$ such that 
\begin{equation}
\label{volterraest}\|y\|_{L^\infty} \leq C \|F\|_{L^\infty}.
\end{equation}
\item [ii)]If $\langle t \rangle^2 K \in L^\infty$  and $ \langle t \rangle^2 F \in L^\infty$, then there exists $C>0$ such that
$$ |\langle t \rangle^2 y\|_{L^\infty} \leq C \|\langle t \rangle^2F\|_{L^\infty}.$$
\item[iii)] If  $\langle t \rangle^3 K \in L^\infty$  and  $\langle t \rangle^\alpha  F \in L^\infty$ for $\alpha \in [2, 3]$ , then 
$$\|\langle t \rangle ^\alpha y \|_{L^\infty} \leq C \| \langle t \rangle^\alpha F \|_{L^\infty}.$$
\end{itemize}
\end{corollary}
 
\begin{proof}
To get i) it suffices to use \eqref{solres} and the Young inequality.

Let us prove ii). We first observe that
$$ t^{1 \over 2} y(t) = K * (t^{1\over 2} y) + \int_{0}^t( t^{1 \over 2} - s^{1 \over 2}) K(t-s) y(s)\, \dd s + t^{1 \over 2} F.$$
By using i), we obtain that
\begin{equation}
\label{volterra1/2} \|t^{1 \over 2} y \|_{L^\infty} \lesssim  \|y \|_{L^\infty} \sup_{t} \int_{0}^t( t^{1 \over 2} - s^{1 \over 2})| K(t-s)|\, \dd s 
+ \| \langle t \rangle^{1 \over 2 } F \|_{L^\infty} \lesssim  \| \langle t \rangle^{1 \over 2 } F \|_{L^\infty} .
\end{equation}
Next, we can write that
$$ t y(t) = K * (ty) + \int_{0}^t  ( t^{1 \over 2} - s^{1 \over 2}) K(t-s) \,s^{1 \over 2} y(s)\, \dd s + t^{1 \over 2} \int_{0}^t( t^{1 \over 2} - s^{1 \over 2}) K(t-s) y(s)\, \dd s + tF.$$
Consequently, by using i) and the assumptions on $K,$ we obtain that 
\begin{multline*} \|t  y \|_{L^\infty} \lesssim 
\sup_{t}\left( \int_{0}^t  {(t-s)^{ 1 \over 2} \over  \langle t-s \rangle^2 }  \, \dd s\right) \, \|\langle t \rangle^{1 \over 2 } y \|_{L^\infty} \\
+ 
\sup_{t} \left(t^{1 \over 2 } \int_{0}^t {(t-s)^{ 1 \over 2} \over  \langle t-s \rangle^2 }  {1 \over \langle s \rangle^{1 \over 2}} \dd s\right)\,
\|t^{1 \over 2} y \|_{L^\infty} +  \| \langle t \rangle F \|_{L^\infty}
\end{multline*}
and by \eqref{volterra1/2},
\begin{equation}
\label{volterrat1}   \|t  y \|_{L^\infty} \lesssim   \| \langle t \rangle F \|_{L^\infty}.
\end{equation}
Note that we have used that
$$ t^{1 \over 2 } \int_{0}^t {(t-s)^{ 1 \over 2} \over  \langle t-s \rangle^2 }  {1 \over \langle s \rangle^{1 \over 2}} \dd s
\lesssim   {1 \over \langle t \rangle^{3 \over 2 }}\int_{{t \over 2}}^t {1 \over \langle s \rangle^{1 \over 2}} \, \dd s
+ \int_{0}^{t \over 2} { 1 \over \langle t -s \rangle^{3 \over 2}} \, \dd s  \lesssim 1.$$
We then estimate $ t^{2}y,$ and for that we write
$$ t^2 y= K * t^2 y + F_{2}$$
where by similar manipulations as above, the source term $F_{2}$ may be estimated as follows
\begin{multline*}
|F_{2}|\lesssim  t^{1\over 2} \left(\left( \langle \cdot \rangle^{1 \over 2} |K|\right) *\left( \langle \cdot \rangle |y|\right) \right)
+  \left( \langle \cdot \rangle^{1 \over 2} |K|\right) *\left( \langle \cdot \rangle^{3 \over 2} |y|\right)
\\+  t \left(  \left(\langle \cdot \rangle^{1 \over 2} |K|\right) *\left( \langle \cdot \rangle^{1 \over 2 } |y|\right)\right) 
+ t^{3 \over 2} \left( \langle \cdot \rangle^{1 \over 2} |K|\right) *( |y|)  + t^2 |F|.
\end{multline*}
By using again i) and \eqref{volterra1/2}, and similar arguments as above, we obtain that
$$ \|t^2  y \|_{L^\infty}
\lesssim \| \langle t \rangle^{3 \over 2} y \|_{L^\infty} + \| \langle t \rangle^2 F \|_{L^\infty}.$$
To conclude, we can use first the interpolation inequality 
$$   \| \langle t \rangle^{3 \over 2} y \|_{L^\infty} \lesssim  \| \langle t \rangle y \|_{L^\infty}^{1 \over 2}  \| \langle t \rangle^{2} y \|_{L^\infty}^{1 \over 2}.$$
Then we apply the Young inequality: for any $\delta >0,$
$$\| \langle t \rangle y \|_{L^\infty}^{1 \over 2}  \| \langle t \rangle^{2} y \|_{L^\infty}^{1 \over 2}\leq \frac{\| \langle t \rangle y \|_{L^\infty}}{2\delta}+ \frac{\delta \| \langle t \rangle^{2} y \|_{L^\infty}}{2}.$$
Choosing $\delta$ small enough, we conclude that
$$   \|t^2  y \|_{L^\infty}
\lesssim \| \langle t \rangle y \|_{L^\infty} + \| \langle t \rangle^2 F \|_{L^\infty}$$
and the result follows by using \eqref{volterrat1}.

To prove iii), we can use the same arguments.  We first write
$$ ty = K *(ty) +  F_{1}$$
with
$$ F_{1} (t)= tF + (tK) * y.$$
Since $tK \in L^1$, we get by using \eqref{volterraest} that
$$ \|t y\|_{L^\infty} \lesssim \|F_{1}\|_{L^\infty} \lesssim  \|\langle t \rangle F \|_{L^\infty}.$$
Next, we write 
$$ t^2y = K* t^2 y + F_{2}, \quad F_{2}=  (tK) * ty + t F^1$$  
and by Young's inequality
$$ \|F_{2}\|_{L^\infty} \lesssim \| tK\|_{L^{1}}\| ty\|_{L^{\infty}} + \|t^2 F \|_{L^\infty} +  \|t ( (tK) * y )\|_{L^\infty} \lesssim  \|t F \|_{L^\infty} + \|t^2 F \|_{L^\infty} +  \|t ( (tK) * y )\|_{L^\infty}. $$
It remains to see that
$$ |( tK) *y |\lesssim  \int_{0}^t   { 1 \over \langle t-s\rangle^2} { 1 \over \langle s \rangle} \|\langle t \rangle y \|_{L^\infty} \dd s
\lesssim { 1 \over \langle t \rangle} \| t y \|_{L^\infty},$$
such that
$$\|F_{2}\|_{L^\infty} \lesssim \|\langle t\rangle ^2 F \|_{L^\infty}.$$
We conclude by using again \eqref{volterraest} that
$$ \| t^2 y \|_{L^\infty} \lesssim  \|\langle t \rangle^2 F \|_{L^\infty}.$$
$t^3 y$ is estimated in the same way as above.      
\end{proof}

We shall then apply the Corollary to the two Volterra equations \eqref{Volt} to prove Theorem \ref{theomagnetiz}, starting with the one satisfied by $\Cc(t).$ Note that by using Proposition \ref{propnoyaux}, and the Penrose criterion \eqref{penrose}, we get that the kernel $K_{\mathcal{C}}$ matches the assumptions of Corollary \ref{corvolterra} iii). To estimate $F_{\mathcal{C}}(t)$ given by \eqref{eq:F}, we can apply Theorem \ref{theodecay} (using the orthogonality condition \eqref{orthocond}) with the functions $ \varphi = \cos(\rX(x,v))$ and  $f =  r^{0}(x,v)$, for which we have $p\geq 0$ and $q=1$. Now without further assumptions this implies that $F_{\mathcal{C}}(t) \lesssim \frac{1}{\langle t \rangle^\alpha}$ with $\alpha=\frac{p + 5}{2}$.  
Therefore the application of Corollary \ref{corvolterra} to the first Volterra equation of \eqref{Volt} yields the estimate on $\Cc(t)$ claimed in Theorem \ref{theomagnetiz}.

In the case of the second Volterra equation of \eqref{Volt}, satisfied by $\Sc(t),$ we  estimate  
$F_{\mathcal{S}}(t)$ given in \eqref{eq:F} by using  Theorem \ref{theodecay} with the functions $\varphi = \sin(\rX(x,v))$ and $f = r^{0}(x,v)$ for which we have $p\geq 0$ and $q=0$. This  yields the estimate
$ |F_{\mathcal{S}}(t) | \lesssim {1 \over \langle t \rangle^\alpha}$ with $\alpha = 2 + \frac{p}{2}$. 
As the kernel  $K_{\mathcal{S}}$ falls under the scope of Corollary \ref{corvolterra} ii), the estimate on $\Sc(t)$ claimed in Theorem \ref{theomagnetiz} follows.

\section{Proof of the scattering result Corollary \ref{weakconv}}

Let us first study the asymptotic behavior of $g,$ and define
$g_{\infty}(x,v)$ by
\begin{equation}
\label{ginfty1}
g_{\infty}(x, v)= 
r^{0}(x, v) + \int_{0}^{+ \infty}\Big(  \mathcal{C}(s) \left\{ \eta, \cos (\rX)\right\}\circ \psi_{s}(x, v)
+   \mathcal{S}(s) \left\{ \eta, \sin (\rX)\right\}\circ \psi_{s}(x, v)
\Big) \, \dd s. 
\end{equation}
Note that the above integral is convergent in $L^1_{x,v}$. Indeed, by using that $\psi_{s}$ is measure preserving and Theorem \ref{theomagnetiz} giving decay  estimates for  $\mathcal{C}(s)$ and $\mathcal{S}(s)$ we get that
$$ \left\|  \mathcal{C}(s) \left\{ \eta, \cos (\rX)\right\}\circ \psi_{s}(x, v)
+   \mathcal{S}(s) \left\{ \eta, \sin (\rX)\right\}\circ \psi_{s}(x, v) \right\|_{L^1_{x,v}}
\lesssim { 1 \over \langle s \rangle^2}.$$
As
$$
g(t,x, v)= 
r^{0}(x, v) + \int_{0}^{t}\Big(  \mathcal{C}(s) \left\{ \eta, \cos (\rX)\right\}\circ \psi_{s}(x, v)
+   \mathcal{S}(s) \left\{ \eta, \sin (\rX)\right\}\circ \psi_{s}(x, v)
\Big) \, \dd s,
$$
this also yields that
\begin{equation}
\label{g-ginfty}
\| g(t)- g_{\infty}\|_{L^1_{x,v}} \lesssim \int_{t}^{+ \infty}{ 1 \over \langle s \rangle^2}\, \dd s \lesssim {1 \over \langle t \rangle},
\end{equation}
which proves the first part of the statement. 
Now, let us study the weak convergence of $r(t,x,v)$. Let us observe that for every test function $\phi(x,v)$, we have by volume preservation that
\begin{multline*} \int_{\T\times \R} r(t,x,v) \phi(x,v) \, \dd x \dd v = \int_{\T\times \R} g(t,x,v) \phi(\psi_{t}(x,v)) \dd x \dd v
\\= \int_{\T\times \R} g_{\infty}(x,v) \phi( \psi_{t}(x,v))\dd x \dd v + \mathcal{O}\left( { 1 \over \langle t \rangle}\right) =: I(t) + + \mathcal{O}\left( { 1 \over \langle t \rangle}\right). 
\end{multline*}
%Therefore, we only need to study the limit when $t \rightarrow + \infty$ of
%$$   I(t)=  \int_{\T \times \R} g_{\infty}(x,v) \phi( \psi_{t}(x,v))\dd x\dd v.$$
%Let us define
%$$ r_{\infty}(x,v)= R_{\infty}(h(x,v))$$
%with
%$$ R_{\infty}(h) = \int_{(-\pi,\pi)} r^{0}(x(\theta, h), v(\theta, h)) \dd \theta,  \quad h \in I_{*}, \, *\in \{\pm, \circ\}.$$
By using the expression \eqref{ginfty1}, and the fact that $\psi_s$ is invertible and preserves the volume, we obtain that
\begin{multline}
\label{korg}
 I(t) = \int_{\T \times \R}  r^{0}(x,v) \phi (\psi_{t}(x,v))\, \dd x \dd v
+ \int_{0}^{+ \infty} \left( \mathcal{C}(s) \int_{\T\times \R} \{\eta, \cos (\rX)\} \phi(\psi_{t-s}(x,v))\, \dd x \dd v \right.\\ \left. +\mathcal{S}(s) \int_{\T \times \R}
\{\eta, \sin (\rX)\}  \phi(\psi_{t-s}(x,v))\, \dd x \dd v \right)\, \dd s.
\end{multline} 
Now, thanks to Theorem \ref{theodecay}, we obtain that 
$$ \int_{\T \times \R} r^{0}(x,v) \phi(\psi_t(x,v)) \dd x \dd v  \rightarrow_{t \rightarrow + \infty}\sum_{* \in \{ \pm,\circ\}} \int_{J_*} (r^{0})_0^*(a) \phi_0^*(a)  \dd a = \int_{\T \times \R} r_{\infty}(x,v) \phi(x,v) \dd x \dd v,$$
with $r_{\infty}(x,v)$ the angle average of $r^0(\theta,a)$, 
$$ 
r_{\infty}(x,v) = (r^0)_0^*(h) =   \frac{1}{2\pi}\int_{(-\pi,\pi)} r^{0}(x(\theta, h), v(\theta, h)) \dd \theta,  \quad h \in I_{*}, \, *\in \{\pm, \circ\}.
$$
Next, we observe that $  \{\eta, \cos (\rX)\}^*_{0}= \{\eta, \sin (\rX)\}^*_{0}= 0$. Consequently, by using again
Theorem \ref{theodecay}, we obtain that
$$ \left| \int_{\T \times \R} \{\eta, \cos (\rX)\} \phi(\psi_{t-s}(x,v))\, \dd x \dd v\right|  + \left| \int_{\T \times \R}
\{\eta, \sin (\rX)\}  \phi(\psi_{t-s}(x,v))\, \dd x \dd v\right|
\lesssim { 1 \over \langle t-s \rangle^2}.$$
Consequently, we find that
\begin{multline*}
\left|  \int_{0}^{+ \infty} \left( \mathcal{C}(s) \int_{\T \times \R} \{\eta, \cos (\rX)\} \phi(\psi_{t-s})\, \dd x \dd v  +\mathcal{S}(s) \int_{\T \times \R}
\{\eta, \sin (\rX)\}  \phi(\psi_{t-s})\, \dd x \dd v \right)\, \dd s \right| \\
\lesssim \int_{0}^{+ \infty} { 1 \over \langle s \rangle^2} { 1 \over \langle t-s \rangle^2} \, \dd s  \lesssim {1 \over \langle t \rangle^2}.
\end{multline*}
%We have thus proven that
%$$ \int_{\T \times \R}  r(t,x,v) \phi(x,v) \, \dd x \dd v \rightarrow_{t \rightarrow + \infty} \int_{\T \times \R} r_{\infty}(x,v) \phi(x,v) \dd x\dd v,$$
and using \eqref{korg} this concludes the proof of corollary \ref{weakconv}.

\section{Penrose condition: Proofs of Theorem \ref{theopenrose} and Proposition \ref{penrosephy}}

\subsection{Proof of Theorem \ref{theopenrose}} 
Let us start with the study of $K_{\mathcal{C}}$ (see \eqref{eq:F}). With the assumption on the profile function $G$, 
Proposition \ref{propnoyaux} shows that ${K_{\mathcal{C}}} \in L^1(\mathbbm{R}_{+})\cap L^2(\mathbbm{R}_{+})$. We have 
$$
\hat{K}_{\Cc}(\xi) = \frac{1}{2\pi}\int_{-\infty}^\infty K_{\Cc}(t) e^{ i t \xi} \dd t
=  \frac{1}{2\pi}\int_{\T \times \R } \int_{0}^\infty  e^{ i t \xi} \{ \eta , \cos ( \rX )\}\cos(\rX \circ \psi_t) \dd x \dd v \dd t, 
$$
which defines a continuous function on the set $\{ \xi \in \C \, | \, \mathrm{Im} (\xi) \leq 0\}$ holomorphic on the set $\{ \mathrm{Im} (\xi) < 0\}$. 

(i) 
As $\frac{\dd}{\dd t}\cos(\rX \circ \psi_t) = \{ \cos(\rX), h_0\} \circ \psi_t$, for $\xi \neq 0$, we have after integration by part and using estimate \eqref{estKCKS}
$$
\hat{K}_{\Cc}(\xi) 
=  - \frac{1}{2i \xi \pi} K_{\Cc}(0) 
  - \frac{1}{i \xi}\frac{1}{2\pi} \int_{0}^\infty  e^{ i t \xi} \int_{\T \times \R } \{ \eta , \cos ( \rX )\}\{ \cos X,h_0\} \circ \psi_t \dd x \dd v \dd t. 
$$
To analyze the second term, 
we can use Theorem \ref{theodecay} with the functions $f= \left\{\eta,\cos(\rX)\right\}$ and $\varphi = \{ \cos X,h_0\} = - \sin(x) v$ for which we have $p = q = 1$. As noted in \eqref{brahms3} the average $f_0(a)$ vanishes and hence the integrand is $\mathcal{O}( \frac{1}{\langle t \rangle^3})$ by using Theorem \ref{theodecay}. This shows that for $\{\mbox{Im}(\xi)\leq 0\}$ and $\xi \neq 0$ we have 
$$
|\hat{K}_{\Cc}(\xi) | \lesssim \frac{1}{|\xi|}. 
$$
Hence there exists $B > 0 $ such that for $|\xi| \geq B$, $|1 - \hat{K}_{\Cc}(\xi)| > \frac12 $. 

(ii) Moreover, as $\eta = G(h_0)$, and as $h_0$ is invariant by the flow, we have that 
\begin{eqnarray*}
K_\mathcal{C}(t) &=& \displaystyle\mathds{1}_{\{t\geq 0\}}\int_{\T \times \R } G' (h_0) \cos(\rX ) \{ h_0 , \cos ( \rX )\} \circ \psi_{-t} \dd x \dd v\\
&=& \displaystyle\mathds{1}_{\{t\geq 0\}} \frac{\dd}{\dd t}
\int_{\T \times \R } G'(h_0) \cos(\rX) \cos ( \rX  \circ \psi_{-t}) \dd x \dd v =  \mathds{1}_{\{t\geq 0\}} \frac{\dd}{\dd t} Q_\mathcal{C}(t),
\end{eqnarray*}
with 
$$
Q_\mathcal{C}(t) = \int_{\T \times \R } G'(h_0) \cos(\rX \circ \psi_t) \cos ( \rX )  \dd x \dd v - Q_0. 
$$
where 
$$
Q_0 = \sum_{* \in \{ \pm,\circ\}} \int_{J_*} G'(h_0(a)) |C_0^*(a)|^2 \dd a, 
$$
where $C_0^*(a) \in \R$ is given in \eqref{CS} (see \eqref{coeffouriercos1} and \eqref{coeffouriercos2} for explicit expressions). By applying Theorem \eqref{theodecay} with the functions $f = G(h_0) \cos X$ and $\varphi = \cos X$, we obtain with this definition of $Q_0$ that 
$$
|Q_\mathcal{C}(t) | \lesssim \frac{1}{\langle t \rangle^3}. 
$$
Hence we can write
$$\hat K_\Cc(\xi) = \int_{0}^{+\infty} e^{- i t \xi} \frac{\dd}{\dd t} Q_\mathcal{C}(t) \dd t=- Q_\mathcal{C}(0) + i  \xi \int_{0}^{+\infty} e^{- i t \xi} Q_\mathcal{C}(t) \dd t,$$
where by the previous estimate the time integral is well defined and uniformly bounded in $\{\mbox{Im}(\xi)\leq 0\}$.
The assumption \eqref{penrose0} can actually be restated as 
$$1 - \hat K_\Cc(0)  = 1 + Q_\mathcal{C}(0)  = \kappa_0 > 0. $$
Hence, by continuity,  there exists $A>0$ such that for $|\xi| \leq A$, we will have $|1 - \hat K_\Cc(\xi) | > \frac{\kappa_0}{2}$.  

(iii) Now let us express $Q_{\Cc}(t)$ in action-angle variables. We have 
$$
Q_\mathcal{C}(t) = - Q_0 + \frac{1}{2\pi} \sum_{*\in \{\circ,\pm\}} \int_{J_{*}\times (-\pi,\pi)} G'(h_0(a)) \cos(\rX\circ \psi_{t}(\theta,a)) \cos (x(\theta,a)) \dd \theta \dd a
$$
By using for $(\theta,a)\in J_{*}\times (-\pi,\pi)$ the identity $\psi_{t}(\theta,a)=\theta+t\omega_{*}(a)$ and the Fourier expansion \eqref{CS} for the cosine function, and by definition of $Q_0$ we infer that 
\begin{eqnarray*}
Q_\mathcal{C}(t) &=& - Q_0 + \frac{1}{2\pi} \sum_{*\in \{\circ,\pm\}} \sum_{\ell,\ell' \in \Z}\int_{J_{*}\times (-\pi,\pi)} G'(a)C_\ell^*(a) C_{\ell'}^*(a) e^{i t \ell \omega_*(a)} e^{i (\ell + \ell') \theta} \dd \theta \dd a
\\
&=& \sum_{*\in \{\pm, \circ\}} \sum_{\ell \neq 0}  
\int_{J_{*}} G'(a)|C_{\ell}^{*}(a)|^2  e^{i t \ell \omega_{*}(a)} \dd a, 
%&=& 2\sum_{*\in \{\pm, \circ\}} \sum_{\ell >0} \int_{J_{*}} G'(h_0(a)) |C_\ell^{*}(a)|^2 \cos( t \ell \omega_{*}(a)) \dd a.
\end{eqnarray*}
%where we have also used the identity $C_{-\ell}^{*}(a)=\overline{C_{\ell}^{*}(a)}.$ 
Now we calculate that for $\mathrm{Im}(\xi) < 0$, 
%$$
%i \xi \int_{0}^{+\infty} e^{- i t \xi} Q_\mathcal{C}(t)\dd t =\sum_{*\in \{\pm, \circ\}} \sum_{\ell>0} \int_{J_{*}}  G'(h_0(a)) |C_\ell^{*}(a)|^2 \frac{\xi}{\xi - \ell \omega_*(a)} \dd a.$$
%Hence we can write 
\begin{eqnarray*}
\hat K_\Cc(\xi) &=& - Q_\mathcal{C}(0) + \sum_{*\in \{\pm, \circ\}} \sum_{\ell \neq 0} \int_{J_{*}}  G'(h_0(a)) |C_\ell^{*}(a)|^2  \frac{\xi}{\xi - \ell \omega_{*}(a)}  \dd a\\
&=&  \sum_{*\in \{\pm, \circ\}} \sum_{\ell\neq 0} \int_{J_{*}}  G'(h_0(a)) |C_\ell^{*}(a)|^2  \left(\frac{\xi}{\xi - \ell \omega_{*}(a)}  -1\right) \dd a\\
&=&  \sum_{*\in \{\pm, \circ\}} \sum_{\ell \neq 0} \int_{J_{*}}  G'(h_0(a)) |C_\ell^{*}(a)|^2  \left(\frac{\ell \omega_{*}(a)}{\xi - \ell \omega_{*}(a)}\right) \dd a.
\end{eqnarray*}
Hence we have for $\xi = \gamma+ i \tau$ with $\tau <0$, 
$$
\mathrm{Re} \, \hat K_\Cc(\xi) = \sum_{*\in \{\pm, \circ\}} \sum_{\ell \neq 0} \int_{J_{*}}  G'(h_0(a)) |C_\ell^{*}(a)|^2  \left(\frac{(\gamma- \ell \omega_*(a))\ell \omega_{*}(a)}{|\gamma - \ell \omega_{*}(a)|^2+\tau^2}\right) \dd a
$$
For a given $-\tau \in [A,B]$, as $G'< 0$ we have that 
\begin{multline*}
\lim_{\gamma\to 0} \mathrm{Re} \, \hat K_\Cc(\xi) = - \sum_{*\in \{\pm, \circ\}} \sum_{\ell \neq 0} \int_{J_{*}}  G'(h_0(a)) |C_\ell^{*}(a)|^2  \left(\frac{\ell^2 \omega_{*}(a)^2}{\ell^2 \omega_{*}(a)^2+\tau^2}\right) \dd a \\ \leq - \sum_{*\in \{\pm, \circ\}} \sum_{\ell \neq 0} \int_{J_{*}}  G'(h_0(a)) |C_\ell^{*}(a)|^2  = - Q_{\Cc}(0)
\end{multline*}
By uniform continuity, this implies that there exists $\varepsilon_0$ such that for $|\gamma|< \varepsilon$ and $-\tau \in [A,B]$, we have 
$| 1 - \hat K_\Cc(\xi) | >\frac{\kappa_0}{2}$. 

(iv) With the same notation as before, we calculate that 
\begin{multline}
\label{IMKC}
\mathrm{Im} \, \hat K_\Cc(\xi) =- \sum_{*\in \{\pm, \circ\}} \sum_{\ell \neq 0} \int_{J_{*}} G'(h_0(a)) |C_\ell^{*}(a)|^2 \tau \ell \omega_{*}(a) \frac{1}{(\gamma - \ell \omega_{*}(a))^2 + \tau^2}\\
=- \sum_{*\in \{\pm, \circ\}} \sum_{\ell > 0} \int_{J_{*}}G'(h_0(a)) |C_\ell^{*}(a)|^2 \tau \ell \omega_{*}(a) \left( \frac{1}{(\gamma - \ell \omega_{*}(a))^2 + \tau^2} - \frac{1}{(\gamma + \ell \omega_{*}(a))^2 + \tau^2}\right) \dd a\\
%&=-& \sum_{*\in \{\pm, \circ\}} \sum_{\ell > 0} \int_{J_{*}} G'(h_0(a)) |C_\ell^{*}(a)|^2 \tau \ell \omega_{*}(a) \left( \frac{(\gamma + \ell \omega_{*}(a))^2 - (\gamma - \ell \omega_{*}(a))^2}{((\gamma - \ell \omega_{*}(a))^2 + \tau^2)((\gamma + \ell \omega_{*}(a))^2 + \tau^2)}\right)\\
=- 4 \gamma \tau \sum_{\underset{\ell > 0}{*\in \{\pm, \circ\}}} \int_{J_{*}}G'(h_0(a)) |C_\ell^{*}(a)|^2  \ell^2 \omega_{*}(a)^2 \left(     \frac{1}{((\gamma - \ell \omega_{*}(a))^2 + \tau^2)((\gamma + \ell \omega_{*}(a))^2 + \tau^2)}\right).
\end{multline}
The coefficients $C_\ell^*(a)$ explicitly given in section \ref{actionangle} are non-zero everywhere (see Propositions \ref{prop74} and \ref{prop710}), and hence the previous term does not vanish when $\gamma\neq 0$ and $\tau \neq 0$. By combining with the previous results, this shows that for all $\varepsilon >0$, there exists $\kappa(\varepsilon)>0$ such that $| 1 - \hat K_{\Cc}(\xi) | \geq \kappa(\varepsilon)$ except possibily if $-\tau \leq \varepsilon$ and $|\gamma| \in [\frac{A}{2},2B]$.

(iv) To conclude, we thus need to study the limit $\tau \to 0$ for $|\gamma| \in [\frac{A}{2},2B]$. By symmetry, we can only consider the case $\gamma>0$ and we know that in this case $\mathrm{Im} \, \hat K_\Cc(\xi) < 0$ is a sum of negative terms. 
In the first sum \eqref{IMKC} giving the expression $\mathrm{Im} \, \hat K_\Cc(\xi)$, we have that for $\ell<0$, $\gamma - \ell\omega_*(a) >\gamma$. Hence for a fixed $\gamma \in  [\frac{A}{2},2B]$ the limit of the corresponding terms when $\tau \to 0$ is $0$ and the only contribution comes from terms for which $\ell>0$. 

We shall use the fact that $\partial_{a} \omega_{*}$ does not vanish on each chart $U_*$ (see Remarks \ref{rem73} and \ref{rem79} in Section \ref{actionangle}). 

Let us consider the upper exterior of the eye, {\em i.e.} $*=+$. We make in the integral on $J_+$ the 
change of variable $u=\ell \omega_+ (a)- \gamma$. Hence when $a\in J_+= (\frac{4}{\pi}\sqrt{M_0},+\infty)$, we have 
$u \in (- \gamma, +\infty)$ by using the formula of Proposition \ref{prop72}. 
Hence we have for $\ell > 0$, 
\begin{multline*}
\int_{J_{+}} G'(h_0(a)) |C_\ell^{+}(a)|^2 \tau \ell \omega_{+}(a) \frac{1}{(\gamma - \ell \omega_{+}(a))^2 + \tau^2}
= \int_{-\gamma}^{+\infty} F_\gamma(\frac{u+ \gamma}{\ell}) \frac{\tau }{u^2 + \tau^2}\dd u\\
= - \int_{-\frac{\gamma}{\tau}}^{+\infty} F_\gamma^+(\frac{|\tau| u+ \gamma}{\ell}) \frac{1}{u^2 + 1}\dd u = - \pi F^+(\frac{\gamma}{\ell}) + R_{\ell,\gamma}( |\tau| ), 
\end{multline*}
where 
\begin{equation}
\label{Fmoins}
F^+(v) := \frac{v}{|\ell \partial_a \omega_+( \omega_+^{-1}(v))|}G'(h_0(\omega_+^{-1}(v))) |C_\ell^+(\omega_+^{-1}(v))|^2.
\end{equation}
Note that in view of \eqref{coeffouriercos1}, for all $\ell$, the coefficient $C_\ell^+$ are non zero, then we have for some $\ell_0$ that for all $\gamma \in [\frac{A}{2},2B]$, $F^+(\frac{\gamma}{\ell_0}) < - \kappa_1$. 
As all the terms in \eqref{IMKC} are non positive, we have  
$$
 \mathrm{Im} \, \hat K_\Cc(\xi)   < \pi  F^*(\frac{\gamma}{\ell_0}) + R_{\gamma,\ell_0}(|\tau|) < - \frac{\kappa_1}{2}
$$
for $-\tau \leq \varepsilon$ small enough. By combination with the previous item, this concludes the proof for $K_\Cc(\xi)$.

Now we consider the case of $K_{\Sc}$. The proof for $\Kc_{\Cc}(\xi)$ is entirely similar, once we have noticed that 
$$
K_{\Sc}(t) = K_\mathcal{S}(t)= \mathds{1}_{\{t\geq 0\}} \frac{\dd}{\dd t} Q_\mathcal{S}(t)
$$
with an expansion in action-angle variables
$$
Q_\mathcal{S}(t) = -2\sum_{*\in \{\pm, \circ\}} \sum_{\ell >0} \int_{J_{*}} G'(h_0(a)) |S_\ell^{*}(a)|^2 \cos( t \ell \omega_{*}(a)) \dd a.
$$
The argument is then identical as for the case of $K_{\Cc},$ since the coefficients $S_{\ell}^{*}(a)$ are non-zero everywhere (see Propositions \ref{prop74} and \ref{prop710}).

\subsection{Proof of Proposition \ref{penrosephy}}
In the case of $K_{\Cc},$ we saw in the previous proof that 
$$1-\hat{K}_{\Cc}(0)> 0 \Leftrightarrow 1+Q_{\Cc}(0) > 0,$$
with 
$$Q_\mathcal{C}(t) = 2\sum_{*\in \{\pm, \circ\}} \sum_{\ell >0} \int_{J_{*}} G'(h_0(a)) |C_\ell^{*}(a)|^2 \cos( t \ell \omega_{*}(a)) \dd a.$$
Now we can use Parseval's identity
$$\sum_{\ell\in\mathbbm{Z}} |C_{\ell}^{*}(a)|^{2}=\frac{1}{2\pi}\int_{(-\pi,\pi)} \cos^{2}(x(\theta,a)) \dd \theta$$
to write that
\begin{equation}
\notag
\begin{split}
Q_{\Cc}(0)&=\sum_{*\in \{\pm, \circ\}} \sum_{\ell \neq0} \int_{J_{*}} G'(h_0(a)) |C_\ell^{*}(a)|^2  \dd a\\
&=\frac{1}{2\pi} \sum_{*\in \{\pm, \circ\}}\int_{J^{*}\times (-\pi,\pi)} G'(h_{0}(a)) \cos^{2}(x(\theta,a)) \dd \theta \dd a - \sum_{*\in \{\pm,\circ\}} \int_{J_{*}}G'(h_{0}(a)) |C_0^{*}(a)|^2 \dd a \\
&= \int_{\T \times \R} G'(h_{0}(x,v)) \cos^{2}(x) \dd x \dd v - \sum_{*\in \{\pm,\circ\}} \int_{J_{*}}G'(h_{0}(a)) |C_0^{*}(a)|^2 \dd a, 
%&= \int_{\T \times \R} \frac{\partial_{v} \eta(x,v)}{v} \cos^{2}(x) \dd x \dd v - \sum_{*\in \{\pm,\circ\}} \int_{J_{*}}\frac{\partial_{a}\eta(a)}{\omega_{*}(a)} |C_0^{*}(a)|^2 \dd a,
\end{split}
\end{equation}
where we have also used area preservation. Hence the condition $1+Q_{\Cc}(0)> 0$ is equivalent to the condition \eqref{eqphy} of Definition \ref{phy}. This proves the result in the case of $K_{\Cc}.$\\
Now in the case of $K_{\Sc},$ we saw in the previous proof that
$$1-\hat{K}_{\Sc}(0)> 0 \Leftrightarrow 1+Q_{\Sc}(0) > 0,$$
with 
$$Q_\mathcal{S}(t) = -2\sum_{*\in \{\pm, \circ\}} \sum_{\ell >0} \int_{J_{*}} G'(h_0(a)) |S_\ell^{*}(a)|^2 \cos( t \ell \omega_{*}(a)) \dd a.$$
Using Parseval's formula as previously, and as $S_{0}^{*}(a)=0$ on each chart, we obtain this time that
$$1+Q_{\Sc}(0) > 0\Leftrightarrow 1-\int_{\T\times \R} G'(h_0(x,v))\sin^{2}(x) \dd x\dd v>0,$$
which is guaranteed by the assumption $G'<0.$

\section{Examples of stable stationary states}

\label{sectionromain}

In this section we study the existence of stationary states of the kind \eqref{cimaros}, and exhibit examples of such states that satisfy the stability hypothesis \eqref{phy}. Let us first make the following comments on the stationary states considered in the introduction. Stationary solutions of \eqref{V-HMF}, are functions $\eta(x,v)$ satisfying $\{\eta,H[\eta]\} = 0$. 
As for all smooth functions $G:\R \to \R$ and $H: \R^2 \to \R$ we have $\{G(H),H\} = G'(H) \{H,H\} = 0$, 
stationary states can be constructed by finding a function $G:\mathbbm{R} \rightarrow \mathbbm{R}$ and a function $\eta$ smooth enough such that the $\eta(x,v)=G\left(H[\eta](x,v)\right)$. Note that for such function, we have 
$$H[\eta] = \frac{v^2}{2} - \Cc[\eta] \cos(x) - \Sc[\eta] \sin(x) = M_0 \cos(x - x_0), $$
where $M_0 \geq 0$ and $x_0$ are real constants attached to the stationary state such that  $M_0 e^{ix_0} = \mathcal{C}[\eta] + i \Sc[\eta]$. 
As \eqref{V-HMF} is invariant by translation $x \mapsto x + x_0$ we can consider the case $x_0 = 0$ (i.e. $M_0 = \mathcal{C}[\eta]$ and $\Sc[\eta] = 0$), and any solutions must satisfy \eqref{eqM0} for this number $M_0$. 

Conversely all functions $\eta(x,v) = G( \frac{v^2}{2} - M_0 \cos(x))$) with 
 $M_0$, called the magnetization of $\eta$, satisfying \eqref{eqM0}  defines a stationary states. 
% 
%
%$$H[\eta] = \frac{v^2}{2} - \Cc[\eta] \cos(x).$$
%Let us indeed specify the argument: we can always define  $M_{0} > 0$ and $x_{0}\in (-\pi,\pi)$ such that
%$$\phi[\eta](x)=M_{0}\cos(x-x_{0}).$$
%For that it suffices to set $M_{0}=\sqrt{\Cc[\eta]^{2}+\Sc[\eta]^{2}},$ and to take $x_{0}$ as the solution of
%$$ M_{0}\cos(x_{0})=\Cc[\eta]\quad \mbox{and} \quad M_{0}\sin(x_{0})=\Sc[\eta].$$
%We have then that
%$$\phi[\eta](x) = M_{0}\cos(x_{0})\cos(x)+M_{0}\sin(x_{0})\sin(x)=M_{0}\cos(x-x_{0}).$$
%By the rotational invariance of the Vlasov-HMF model \eqref{V-HMF}, we may then assume that, up to the translation $x\to x+x_{0},$ we have
%$$\phi[\eta](x)=M_{0}\cos(x) ,\quad \mbox{such that} \quad \Cc[\eta]=M_{0}\quad \mbox{and} \quad \Sc[\eta]=0.$$
%$M_{0}$ is usually called the magnetization of $\eta,$ and $\eta$ solves equation \eqref{statstates} 
% (Let us seek solutions under 
%
%From now on, we fix a triplet $(\eta,G,M_{0})$ such that
%\begin{equation}
%\label{cimaros}
%\eta(x,v)=G\left(h_{0}(x,v)\right), \quad h_{0}(x,v)=\frac{v^{2}}{2}-M_{0}\cos(x), \quad M_{0}>0.
%\end{equation}
% 

\subsection{Sufficient conditions of existence and stability}

The following Proposition provides a sufficient condition on the function $G$ such that an inhomogeneous state of the kind \eqref{cimaros} exists.

\begin{proposition}
\label{general}
Let $G:[-e,+\infty[\to \mathbbm{R}_{+}$ be a $C^{1}$ function such that $G,G'\in L^{1}([-e,+\infty[).$ Assume that there exists $\zeta>0$ such that
\begin{equation}
\label{first}
\int_{\T\times \R} G\left(\frac{v^{2}}{2}-\zeta \cos(x)\right) \cos(x) \dd x \dd v \geq \zeta,
\end{equation}
and that
\begin{equation}
\label{second}
1+\int_{\T\times \R} G'\left(\frac{v^{2}}{2} \right) \cos^{2}(x) \dd x\dd v> 0.
\end{equation}
Then there exists a solution $M_{0}>0$ to the equation
$$M_{0}=\int_{\T \times \R} G\left(\frac{v^{2}}{2}-M_{0} \cos(x)\right) \cos(x) \dd x \dd v.$$
In particular, $\eta(x,v)=G\left(\frac{v^{2}}{2}-M_{0} \cos(x)\right)$ is an inhomogeneous stationary solution of \eqref{V-HMF}.
\end{proposition}

\begin{proof}
Consider the function
$$F(z)= \int_{\T \times \R} G\left(\frac{v^{2}}{2}-z \cos(x)\right) \cos(x) \dd x \dd v-z.$$
We have $F(0)=0$ (as the cosine function has average $0$ on $(-\pi,\pi)$), and the hypothesis imply that $F(\zeta)\geq 0$ and $F'(0)<0.$ Hence either $F(\zeta)=0$ and the proof is done, or $F(\zeta)>0,$ and the intermediary value Theorem shows that there exists $M_{0}\in (0,\zeta)$ such that
$ F(M_{0})=0.$
\end{proof}

The next Proposition gives a sufficient condition to fulfill the stability assumption of definition \ref{phy}, which is moreover independent of the angle-action variables. 

\begin{proposition}
\label{general3}
Let $G:[-e,+\infty[\to \mathbbm{R}_{+}$ be a $C^{1}$ function such that $G,G'\in L^{1}([-e,+\infty[),$ and $\eta$ be defined by \eqref{cimaros} with $M_{0}>0.$ Assume that $G'<0,$ and that $\eta = G(h_0)$ satisfies 
$$1+\int_{\T\times \R} G'(h_0(x,v))\cos^{2}(x)\dd x\dd v - \frac{\left(\displaystyle\int_{\T\times \R} \cos(x)G'(h_0(x,v))\dd x\dd v\right)^{2}}{\displaystyle\int_{\T\times \R} G'(h_0(x,v))\dd x\dd v}>0.$$
Then $\eta$ is stable in the sense of definition \ref{phy}.
\end{proposition}

\begin{proof}
By using the symplectic variable $(\theta,a)$, we have 
$$\frac{\left(\displaystyle\int_{\T\times \R} \cos(x)G'(h_0(x,v))\dd x\dd v\right)^{2}}{\displaystyle\int_{\T\times \R} G'(h_0(x,v))\dd x\dd v}=\frac{\left(\displaystyle\sum_{*\in \{\circ,\pm\}}\displaystyle\int_{J_{*}} C_{0}^{*}(a)G'(h_0(a))\dd a\right)^{2}}{\displaystyle\sum_{*\in \{\circ,\pm\}}\displaystyle\int_{J_{*}} G'(h_0(a))\dd a}.$$
Hence it is enough to check that
\begin{equation}
\label{check}
\sum_{*\in\{\circ,\pm\}}\int_{J_{*}} C_{0}^{*}(a)^{2} G'(h_0(a))\dd a \leq \left(\sum_{*\in\{\circ,\pm\}}\int_{J_{*}} C_{0}^{*}(a)G'(h_0(a))\dd a\right)^{2}\left(\sum_{*\in\{\circ,\pm\}}\int_{J_{*}} G'(h_0(a)) \dd a\right)^{-1}.
\end{equation}
Now for $*\in\{\circ,\pm\},$ we define on $J_{*}$ a function
$$F(a)=\left(\sum_{*\in\{\circ,\pm\}}\int_{J_{*}} G'(h_0(a))\dd \alpha\right)^{-1} G'(h_0(a)) >0,$$
which is positive, since $G'<0$ and of global integral $1$. The Cauchy-Schwarz inequality implies then that for all $*\in\{\circ,\pm\}$
$$ \left(\sum_{*\in\{\circ,\pm\}}\int_{J_{*}} C_{0}^{*}(a)F(a)\dd a\right)^{2} \leq \sum_{*\in\{\circ,\pm\}}\int_{J_{*}} C_{0}^{*}(a)^{2} F(a)\dd a. 
$$
Multiplying both sides of the inequality by the real number $\sum_{*\in\{\circ,\pm\}}\int_{J_{*}} G'(h_0(a)) \dd a$ which is negative
we obtain that \eqref{check} is true, and the proof is done.
\end{proof}

\subsection{Example of stable stationary states: Maxwell-Boltzmann distributions}

Here we study the case where the function $G$ is an exponential. As we consider averages of $G$ against cosine functions, we introduce the modified Bessel functions of the first kind: 
$$I_{n}(z)=\frac{1}{\pi} \int_{0}^{\pi} e^{z\cos(x)}\cos(nx) \dd x=\int_{\mathbbm{T}} e^{z\cos(x)}\cos(nx) \dd x.$$
We shall use the following assymptotics (see formulae 9.6.10 and 9.7.1 of \cite{abramo}):
\begin{equation}
\label{besselDL}
I_{n}(z)=\left(\frac{z}{2}\right)^{n}\left[\frac{1}{n!}+\frac{z^{2}}{4(n+1)!}+\mathcal{O}(z^{4})\right] \quad \mbox{when} \quad z\to 0,
\end{equation}
\begin{equation}
\label{besselDL2}
I_{n}(z)=\left(\frac{e^{z}}{\sqrt{2\pi z}}\right)\left[1-\frac{4n^{2}-1}{8z}+\mathcal{O}\left(\frac{1}{z^{2}}\right)\right] \quad \mbox{when} \quad z\to +\infty.
\end{equation}
We shall also use the following result

\begin{proposition}[{\cite{Bessel}}]
\label{bessinIn}
For all $n\in\mathbbm{N}$ and $z\in \R,$ we have
$$z\frac{I_{n}'(z)}{I_{n}(z)} <\sqrt{z^{2}+n^{2}} \quad \mbox{and} \quad \frac{I_{n+1}(z)}{I_{n}(z)} > \frac{\sqrt{(n+1)^2 +z^{2}} - (n+1)}{z}.$$
\end{proposition}

We shall prove the following result, which shows there exists inhomogeneous Maxwell-Boltzmann distributions
$$\eta(x,v)=\alpha e^{-\beta h_{0}(x,v)}$$
that are stationary solutions of \eqref{V-HMF} of the kind \eqref{cimaros} (see also \cite{Chavanis3}).

\begin{proposition}
\label{gauss}
Let $\alpha, \beta\in \R_{+}^{*},$ and $G(s)=\alpha e^{-\beta s}.$ Then if $\alpha \sqrt{\beta} <\frac{2}{\sqrt{2\pi}},$ $G$ satisfies the conditions \eqref{first} and \eqref{second}. 
\end{proposition}

\begin{proof}
We have for any $z>0$
$$\int_{\T\times \R} G\left(\frac{v^{2}}{2}-z\cos(x)\right)\cos(x)\dd x\dd v = \alpha \int_{\T\times \R} e^{-\beta\frac{v^{2}}{2}} e^{\beta z\cos(x)}\cos(x)\dd x\dd= \alpha \sqrt{\frac{2\pi}{\beta}} I_{1}(\beta z),$$
and \eqref{first} is clearly guaranteed by \eqref{besselDL2} for $z$ sufficiently large. \\
Using the first formula of \eqref{classictrigo}
$$1+\int_{\T\times \R} G'\left(\frac{v^{2}}{2} \right) \cos^{2}(x) \dd x\dd v = 1-\alpha \sqrt{\beta} \sqrt{2\pi} \left[\frac{I_{0}(0)+I_{2}(0)}{2}\right]= 1- \alpha \sqrt{\beta} \frac{\sqrt{2\pi}}{2}.$$
That last quantity is positive when $\alpha \sqrt{\beta} <\frac{2}{\sqrt{2\pi}},$ and this concludes the proof.
\end{proof}

Now we shall prove that the inhomogeneous states given by Proposition \ref{gauss} are stable in the sense of definition \ref{phy}

\begin{proposition}
\label{stab-gauss}
Let $\eta$ be a stationary solution of \eqref{V-HMF} given by
$$\eta(x,v)=\alpha e^{-\beta\left(\frac{v^{2}}{2}-M_{0}\cos(x)\right),}$$
with $\alpha,\beta,M_{0} \in \R_{+}^{*},$ $\alpha\sqrt{\beta}<\frac{2}{\sqrt{2\pi}},$ and $M_0$ given by Proposition \ref{general}, and satisfying
\begin{equation}
\label{M0gauss}
M_{0}=\alpha\int_{\T\times \R} e^{-\beta \left(\frac{v^{2}}{2}-M_{0}\cos(x)\right)} \cos(x) \dd x\dd v.
\end{equation}
Then $\eta$ is stable in the sense of definition \ref{phy}.
\end{proposition}

\begin{proof}
We shall prove that the assumptions of Proposition \ref{general3} are fulfilled, which will imply the result. First, we have
\begin{equation}
\notag
\begin{split}
\int_{\R \times \T}G'(h_0(x,v))\cos^{2}(x)\dd x\dd v &=-\alpha \beta \int_{\T \times \R} e^{-\beta \frac{v^{2}}{2}} e^{\beta M_{0}\cos(x)} \cos^{2}(x) \dd x \dd v\\
&= -\alpha\sqrt{\beta} \frac{(2\pi)^{1/2}}{2}\left[I_{0}(\beta M_{0}) + I_{2}(\beta M_{0})\right],
\end{split}
\end{equation}
using the first formula of \eqref{classictrigo}. Then have also
$$\int_{\T\times \R} \cos(x) G'(h_0(x,v)) \dd v\dd x= -\alpha \beta \int_{\T\times \R} \cos(x) e^{-\frac{v^{2}}{2}} e^{\beta M_{0}\cos(x)} \dd x\dd v=-\alpha\sqrt{\beta} (2\pi)^{1/2}I_{1}(\beta M_{0})$$
and
$$\int_{\T\times \R} G'(h_0(x,v))\dd v\dd x = -\alpha\sqrt{\beta} (2\pi)^{1/2}I_{0}(\beta M_{0}).$$
Hence, by Proposition \ref{general3}, it is sufficient to verify that
$$1 -\frac{\alpha \sqrt{\beta}}{2}(2\pi)^{1/2}\left[I_{0}(\beta M_{0})+I_{2}(\beta M_{0})\right] + \alpha \sqrt{\beta}(2\pi)^{1/2} \frac{{I_{1}(\beta M_{0})}^{2}}{I_{0}(\beta M_{0})}>0.$$
Note that
$$I_{0}(\beta M_{0})+I_{2}(\beta M_{0})= 2 I_{1}'(\beta M_{0}).$$
Moreover, $M_0$ satisfies \eqref{M0gauss} which can be written 
$$M_{0}=\alpha \sqrt{\frac{2\pi}{\beta}} I_{1}(\beta M_{0}) \quad\mbox{which implies}\alpha \sqrt{\beta}=\frac{\beta M_{0}}{(2\pi)^{1/2}I_{1}(\beta M_{0})}.$$
Hence it is sufficient to show that
$$1-\frac{\beta M_{0}I_{1}'(\beta M_{0})}{I_{1}(\beta M_{0})}+\frac{\beta M_{0}I_{1}(\beta M_{0})}{I_{0}(\beta M_{0})}>0.$$
But Proposition \ref{bessinIn} implies that
$$\frac{\beta M_{0}I_{1}(\beta M_{0})}{I_{0}(\beta M_{0})}>\sqrt{1+(\beta M_{0})^{2}}-1 \quad\mbox{and}\quad 
-\frac{\beta M_{0}I_{1}'(\beta M_{0})}{I_{1}(\beta M_{0})}>-\sqrt{1+(\beta M_{0})^{2}}.$$
Hence 
$$1-\frac{\beta M_{0}I_{1}'(\beta M_{0})}{I_{1}(\beta M_{0})}+\frac{\beta M_{0}I_{1}(\beta M_{0})}{I_{0}(\beta M_{0})}>0,$$
and the proof is done.
\end{proof}

\section{Action-angle variables \label{actionangle}}

In this section we shall recall how angle-action variables are constructed on each chart $U_{*}.$ It will involve elliptic integrals and Jacobi's elliptic functions, whose definitions and main properties are summarized in the following subsection. 

\subsection{Elliptic integrals, elliptic functions, and elliptic trigonometry}

For $k \in (0,1)$ and $\phi \in (-\pi/2,\pi/2),$ we define the incomplete elliptic integrals by
$$E(\phi,k) = \int_0^\phi \sqrt{1 - k^2 \sin(y)} \dd y \quad \mbox{and} \quad F(\phi,k) = \int_0^\phi \frac{1}{\sqrt{1 - k^2 \sin(y)}} \dd y$$ 
and the complete elliptic integrals by
$$\bE(k) = E\left(\frac{\pi}{2},k\right) \quad \mbox{and} \quad \bK(k) = F\left(\frac{\pi}{2},k\right).$$
We will use the following standard notations: The complementarity modulus $k' = \sqrt{1 - k^2}$, $\bK'(k) = \bK(k')$
and Jacobi's nome
$$q(k) = \exp( -\pi \bK'(k) / \bK(k)). $$
We collect below some useful results for these functions.

\begin{proposition}
\label{propEKq}
The functions $\bE(z)$, $\bK(z)$ and $q(z)$ extend as analytic functions of $z^2$ for $|z|< 1,$ satisfying $\bE(0) = \bK(0) = \frac{\pi}{2}$ and $q(0) = 0$, and we have 
\begin{equation}
\label{rachm0}
\begin{split}
\bE(z) &\sim \frac{\pi}{2} \left( 1 - \frac{1}{4} z^2\right) \quad \mbox{when}\quad z \to 0,\\
\bK(z) &\sim \frac{\pi}{2} \left( 1 + \frac{1}{4} z^2\right)  \quad  \mbox{when}\quad z \to 0,\\
q(z) &\sim  \frac{z^2}{16}  \quad \mbox{when}\quad z \to 0. 
\end{split}
\end{equation} 
Moreover, these functions have logarithmic singularities in $z = 1$:
\begin{equation}
\label{rachm1}
\begin{split}
\bE(z) &\sim 1 - \displaystyle\frac12 (1 - z) \log (1-z)\quad \mbox{when}\quad z \to 1,\\
\bK(z) &\sim - \displaystyle \frac12\log (1-z)  \quad  \mbox{when}\quad z \to 1,\\
q(z) &\sim  \displaystyle 1 + \frac{\pi^2}{\log(1-z)}  \quad \mbox{when}\quad z \to 1.
\end{split}
\end{equation}
More precisely, for all $n \geq 1$ there exists constants $C_n$ such that
\begin{equation}
\label{rachm}
\begin{split}
&\Norm{(1 - z)^{n}\partial_{z}^{n+1}\bE(z)}{L^{\infty}(\frac{1}{2},1)} \leq C_n,\\
&\Norm{(1 - z)^{n}\partial_{z}^n\bK(z)}{L^{\infty}(\frac{1}{2},1)} \leq C_n, \\
&\left\| \log(1-z)^{2}(1 - z)^{n}\partial_{z}^n\Big(\frac{1}{\bK(z)}\Big)\right\|_{L^{\infty}(\frac{1}{2},1)} \leq C_n,\\ 
&\left\| \log(1-z)^{2}(1 - z)^{n}\partial_{z}^n q(z)\right\|_{L^{\infty}(\frac{1}{2},1)} \leq C_n,\\ 
&\left\|(1 - z)^{n}\partial_{z}^n\Big(\frac{1}{1 - q(z)}\Big) \right\|_{L^{\infty}(\frac{1}{2},1)} \leq C_n. 
\end{split}
\end{equation}
\end{proposition}

\begin{proof}
The statements of \eqref{rachm0}, \eqref{rachm1} and \eqref{rachm} concerning $\bE(z)$ and $\bK(z)$ are consequences of the power series expansions (900.00) and (900.05) of \cite{Elliptic1} for the function $\bK(z)$, and (900.07) and (900.10) for the function $\bE(z)$ . In particular, near $z = 1$, we have 
$$\bK(z) = \log(4/z') K_1(z') + K_2(z')$$
where $K_1$ and $K_2$ are smooth functions of $(z')^2 = 1 - z^2$ and $K_1(0) = 1$. 
In other words, we have for $ z \in (1/2,1)$, 
$$\bK(z) = \log(1 - z) A(z) + B(z) >0$$
with $A$ and $B$ smooth functions of $z^{2}$ and $A(1) = -\frac12$. The estimates on $1/\bK(z)$ follow from this formula.\\
The first statement \eqref{rachm0} concerning the function $q(z)$ is a consequence of formula (900.05) of \cite{Elliptic1}. The second \eqref{rachm1} of the expansion 
$$q(z) = \exp( -\pi \bK(\sqrt{1 - z^2}) / \bK(z)) = \sum_{n\geq 0} \frac{(-1)^n}{n!} \left(\frac{\pi \bK(\sqrt{1 - z^2}) }{ \bK(z) } \right)^n ,$$
that holds near $z=1.$ Note that as $\bK(z)$ is an analytic function $z^2$, $\bK(\sqrt{1 - z^2})$ is an analytic function of $z$ for $|z| < 1$ which is bounded as well as its derivatives in the vicinity of $z = 1$. This completes the proof of \eqref{rachm}.
\end{proof}

The Jacobi elliptic functions are then defined as follows: first, we define the amplitude $\am(u,k)$ by the formula
\begin{equation}
\label{fuk}
F(\am(u,k),k) = u.
\end{equation}
The first Jacobi elliptic function is then
\begin{equation}
\label{snuk}
\sn(u,k) = \sin(\am(u,k)). 
\end{equation}
The second and third Jacobi elliptic functions are defined by the formulae
$$\cn(u,k) = \sqrt{1 - \sn^2(u,k)} \quad \mbox{and}\quad \dn(u,k) = \sqrt{1 - k^2 \sn^2(u,k)}. $$
We have the following Fourier series for these functions (see formulae (908.00)--(908.03) of \cite{Elliptic1}):
\begin{equation}
\label{amuk}
\begin{split}
&\am (u,k) = \frac{\pi u}{2 \bK(k)} + 2 \sum_{m = 0}^\infty \frac{q(k)^{m+1}}{(m+1) ( 1+ q(k)^{2(m+1)})}
\sin\left((m+1) \frac{\pi u}{\bK(k)}\right), \\
&\sn(u,k) = \frac{2\pi}{k \bK(k)}\sum_{m = 1}^{\infty} \frac{q(k)^{m - \frac{1}{2}}}{1 - q(k)^{2m - 1}} \sin\left( (2 m -1) \frac{\pi u}{2 \bK(k)}\right),\\
&\cn(u,k) = \frac{2\pi}{k \bK(k)}\sum_{m = 1}^{\infty} \frac{q(k)^{m - \frac{1}{2}}}{1 + q(k)^{2m - 1}} \cos\left( (2 m -1) \frac{\pi u}{2 \bK(k)}\right),\\ 
&\dn (u,k) = \frac{\pi}{2 \bK(k)} + \frac{2\pi}{\bK(k)} \sum_{m = 1}^\infty \frac{q(k)^{m}}{  1+ q(k)^{2m}}
\cos\left(m \frac{\pi u}{\bK(k)}\right).
\end{split}
\end{equation}
The following formulae will also be useful (see (2.14), (2.24) in \cite{Elliptic3}): 
\begin{equation}
\label{Milne}
\begin{split}
&\sn^2(u,k) = \frac{\bK(k) - \bE(k)}{k^2 \bK(k)} -\frac{2\pi^2}{k^2 \bK(k)^2} \sum_{m = 1}^\infty \frac{m q(k)^{m}}{1 - q(k)^{2m}} \cos\left( m  \frac{\pi u}{\bK(k)}\right) \\
&\sn(u,k) \cn(u,k) = \frac{2\pi^2}{k^2 \bK(k)^2} \sum_{m = 1}^\infty \frac{m q(k)^{m}}{1 + q(k)^{2m}} \sin\left( m  \frac{\pi u}{\bK(k)}\right)\\
&\sn(u,k) \dn(u,k) = \frac{\pi^2}{k \bK(k)^2} \sum_{m = 1}^\infty \frac{(2 m - 1) q(k)^{m - \frac{1}{2}}}{1 + q(k)^{2m - 1}} \sin\left( (2 m - 1)  \frac{\pi u}{2\bK(k)}\right).
\end{split}
\end{equation}
We shall also need the following elliptic trigonometry identities (see formulae 120.02, 122.00, 122.03 in \cite{Elliptic1})
\begin{equation}
\label{trigo1}
\begin{split}
&\sn(-u,k)=-\sn(u,k), \quad \cn(-u,k)=\cn(u,k),\quad \dn(-u,k)=\dn(u,k),\\
&\sn(u+\bK(k),k)=\sn(\bK(k)-u,k), \quad \cn(u+\bK(k),k)=-k'\frac{\sn(u,k)}{\dn(u,k)},
\end{split}
\end{equation}
of which two straightforward consequences are the following equalities
\begin{equation}
\label{trigo2}
-\sn(u-\bK(k),k)=\sn(u+\bK(k),k) \quad \mbox{and}\quad -\cn(u-\bK(k),k)=\cn(u+\bK(k),k).
\end{equation}
Finally, we recall for completion some classical trigonometry identities which we often use: for a real number $z,$
\begin{equation}
\label{classictrigo}
\begin{split}
&2\cos^{2}(z)=1+\cos(2z),\quad 2\sin^{2}(z)=1-\cos(2z),\quad \sin(2z)=2\sin(z)\cos(z),\\
&\arcsin(\cos(z))=\sqrt{1-z^{2}}.
\end{split}
\end{equation}

\subsection{Action-angle variables on $U_+$ or $U_-$.} 

We will use the following notations: $\epsilon_{+} = 1,$ and $\epsilon_{-} = -1$. The action-angle coordinates are constructed on $U_{\pm}$ as follows.

\begin{proposition}
\label{prop72}
For $*\in \{\pm\},$ there exists a symplectic change of variable $(x,v) \mapsto (\psi,h)$ from $U_{*}$ to the set 
$$V_* := \{(\psi,h) \in \R^2, | h \in (M_0,+\infty), \, \psi \in (-r_{*}(h),r_{*}(h))\}, $$
with 
$$r_{*}(h) = \frac{1}{ k(h) \sqrt{M_0}} \bK\left(\frac{1}{k(h)}\right), \quad \mbox{where} \quad k(h) = \sqrt{\frac{h+M_0}{2M_0}}, $$
such that the flow of the pendulum in the variables $(\psi,h)$ is $h(t) = h(0)$ and $\psi(t) = t + \psi(0)$. 
There exists then a second symplectic change of variables $(\psi,h) \mapsto (\theta,a)$ from $V_*$ to 
$$W_{*}=\left\{(\theta,a) \in \R^2, | a \in J_{*}=\left(\frac{4}{\pi}\sqrt{M_0},+\infty\right), \, \theta \in (-\pi,\pi)\right\}, $$
such that 
$$\left\{
\begin{array}{lcl}
a(h) &=& \displaystyle \frac{4}{\pi}  k(h) \sqrt{M_0} \bE\left(\frac{1}{k(h)}\right)\\[2ex]
\theta(\psi,h) &=& \omega_{*}(h) \psi
\end{array}\right.
\quad\mbox{with}\quad 
\omega_{*}(h) = \frac{\pi k(h) \sqrt{M_0}} {\bK\left(\frac{1}{k(h)}\right)},$$
and so that the flow of the pendulum in the variables $(\theta,a)$ is $a(t) = a(0)$ and $\theta(t) = t \omega_{*}(a) + \theta(0) .$\\
Moreover, we can easily express $x$ and $v$ as functions of the variables $(\theta,h)$ with the formulae
\begin{eqnarray}
x(\theta,h) &=& \epsilon_* 2\, \am \left(\frac{1}{\pi}\bK\left(\frac1{k(h)}\right) \theta , \frac{1}{k(h)}\right), \\
v(\theta,h) &=&  \epsilon_*  2k(h) \sqrt{M_0}\, \dn \left(\frac{1}{\pi}\bK\left(\frac1{k(h)}\right) \theta , \frac{1}{k(h)}\right).
\end{eqnarray}
\end{proposition}
\begin{remark}
\label{rem73}
Note that we can check directly from the formulae that $\omega_*(h)$ is increasing, and as $h$ is a strictly increasing function of $h$ (see\eqref{omg1}), we have that $\omega_*(a)$ is decreasing, and $\partial_a \omega_\circ(a) >0$, $a \in J_*$. 
\end{remark}
\begin{proof} The construction is classic by using generating functions. 
Setting $h(x,v) = \frac{v^2}{2} - M_0 \cos(x)$, we have on $U_{*}$ 
$$v (x,h) = \epsilon_*\sqrt{2(h + M_0 \cos (x))}.$$
Note that $v(x,h)=\partial_{x}S(x,h), $ where 
\begin{multline*}
S(x,h) = \epsilon_* \int_{0}^x \sqrt{2(h + M_0 \cos (y))} \dd y 
= \epsilon_* 2 \sqrt{2(h + M_0)}\int_{0}^{x/2} \sqrt{1 - \frac{2M_0}{h + M_0} \sin^2 ( y)} \dd y\\
= \epsilon_* 4 \sqrt{M_0} k(h)\int_{0}^{x/2} \sqrt{1 -  \frac{\sin^2 ( y)}{k(h)^2}} \dd y
=  \epsilon_* 4 k(h)\sqrt{M_0} E\left(\frac{x}{2},\frac{1}{k(h)}\right). 
\end{multline*}
We define then the variable $\psi(x,h)$ by
$$\psi(x,h) = \partial_{h}S(x,h) = \epsilon_* \int_{0}^x \frac{1}{\sqrt{2(h + M_0 \cos (y))}} \dd y=  \epsilon_* \frac{1}{ k(h) \sqrt{M_0}} F\left(\frac{x}{2},\frac{1}{k(h)}\right). $$
By construction, the variables $(\psi,h)$ are symplectic, and $S$ is the mixed-variable generating function (see formula (5.5) p197 of \cite{HLW}). We have by the above formulae
$$\dot \psi(t) = \dot x(t) \partial_{x}\psi(x(t),h(x(t),v(t))) = v(t) \frac{\epsilon_{*}}{\sqrt{2(h(x(t),v(t)) + M_{0} \cos(x(t)))}}=1,$$
and the preservation of the hamiltonian reads $\dot h=0,$ such that the flow associated with $h_{0}$ is in these new variables $\dot{\psi}(t)=1, \dot{h}(t)=0.$\\
Setting now
$$r_{*}(h) = \int_0^\pi\frac{1}{ \sqrt{2(h + M_0 \cos (y))} } \dd y = \frac{1}{ k(h) \sqrt{M_0}} \bK\left(\frac{1}{k(h)}\right),$$
we have $\psi \in [-r_{*}(h),r_{*}(h)],$ and the orbits of the flow are periodic with period $2r_{*}(h),$ which is unsatisfying for us. Thus the second step is to perform another transformation which shall force all trajectories to evolve in a torus. If we define 
\begin{equation}
\label{omg1}
g_{*}(h) =  \frac{1}{\pi} {r_{*}(h)} = \frac{1}{\pi} \int_0^{\pi} \frac{1}{\sqrt{2(h + M_0 \cos (y))}} \dd y =   \frac{\partial}{\partial h} a_{*}(h)  > 0,
\end{equation}
with 
$$a_{*}(h) = \frac{1}{\pi} \int_0^{\pi} \sqrt{2(h + M_0 \cos (y))} \dd y = \frac{4}{\pi}  k(h) \sqrt{M_0} \bE\left(\frac{1}{k(h)}\right),$$
and set $\theta (x,h) = \frac{1}{g(h)} \psi(x,h),$ then for each $*\in \{\pm\},$ the variables $(\theta,a)$ belong to $J_{*}\times (-\pi,\pi)$ and are symplectic. Moreover the flow reduces to 
\begin{equation}
\label{omcr1}
\dot{\theta}(t) = \frac{\dot{\psi}(t)}{g_{*}(h)}=\omega_{*}(a) \quad \mbox{and} \quad \dot{a}(t) = 0, \quad \mbox{with} \quad 
\omega_{*}(a_{*}(h)) = \frac{1}{g_{*}(h)} = \frac{1}{\partial_h a_{*}(h)}.
\end{equation}
Note that we have 
$$g_{*}(h) = \frac{1}{\pi k(h) \sqrt{M_0}} \bK\left(\frac{1}{k(h)}\right), $$
which gives the formula for $\omega_{*}(h).$ \\
We can express $x$ in terms of $\theta$ and $h$ using the formula
$$\theta = \epsilon_* \frac{\pi}{\bK\left(\displaystyle\frac{1}{k(h)}\right)}F\left(\frac{x}2,\frac{1}{k(h)}\right). $$
Using the definition of the first Jacobi elliptic function \eqref{fuk} and \eqref{snuk}, we obtain
$$
\sin(x(\theta,h)/2) = \epsilon_*\sn\left(\frac{1}{\pi}\bK\left(\frac1{k(h)}\right) \theta , \frac{1}{k(h)}\right)
$$
and hence 
$$x(\theta,h) = \epsilon_* 2\, \am \left(\frac{1}{\pi}\bK\left(\frac1{k(h)}\right) \theta , \frac{1}{k(h)}\right).$$
Moreover, we have 
\begin{multline*}
v(\theta,h) = \epsilon_* \sqrt{2(h + M_0 \cos (x))} 
= \epsilon_* \sqrt{2(h+M_0)} \sqrt{1 - \frac{1}{k(h)^2} \sin^2(x(\theta,h)/2)}  \\
= \epsilon_* 2k(h) \sqrt{M_0}\, \dn \left(\frac{1}{\pi}\bK\left(\frac1{k(h)}\right) \theta , \frac{1}{k(h)}\right), 
\end{multline*}
and this ends the proof. 
\end{proof}

We consider now the asymptotics of these functions. Note that the variables $(\theta,h)$ are not symplectic, but we will use them to examine these asymptotics. Note moreover that the change $a_{*}(h)$ defined above allows to compute easily integrals in $(\theta,h)$ by using $\dd a = g_{*}(h) \dd h = \frac{1}{\omega_{*}(h)} \dd h$.

\begin{proposition}
\label{prop73}
For $*\in\{\pm\},$ the functions $\omega_{*}(h)$, $x(\theta,h)$ and $v(\theta,h)$ are analytic for $\theta \in (-\pi,\pi)$ and $h \in (M_0,+\infty)$.\\
The function $\omega_{*}$ exhibits the following asymptotic behavior $\omega_*(h)$ is stricly increasing and we have
\begin{equation}
\label{prop73-1}
\omega_{*}(h) \sim \sqrt{2h} \quad \mbox{when}\quad h \to +\infty,\quad \mbox{and} \quad 
\omega_{*}(h) \sim  \frac{2\pi \sqrt{M_0}}{\log\left(\frac{1}{(h-M_0)}\right)}\quad \mbox{when} \quad h \to M_0^+,
\end{equation}
and there exists constants, $C_r$, $\omega_r \neq 0$ and $\alpha_r \neq 0$ such that for all $r \geq 1$, 
\begin{equation}
\label{prop73-2}
\begin{split}
&\left\| \sqrt{h} \Big[ h^{ - \frac{1}{2} + r} \partial_h^r \omega_{*}(h) - \omega_r\Big]\right\|_{L^\infty(2M_0,+\infty)} \leq C_r, \\ 
&\left\| \log(h-M_{0})^{2}\left[\log(h-M_{0})^{2}(h - M_0)^{ r} \partial_h^r \omega_{*}(h) - \alpha_r\right]\right\|_{L^\infty(M_0,2M_0)} \leq C_r. 
\end{split}
\end{equation}
The change of variable $(h,\theta) \mapsto (x,v)$ satisfies the following estimates: for large $h$ it converges towards the "identity" in the sense that
\begin{equation}
\label{prop73-3}
\begin{split}
&x(\theta,h) \sim \epsilon_* \theta \quad \mbox{when}\quad h \to +\infty,\\
&v(\theta,h) \sim \epsilon_* \sqrt{2h}\quad \mbox{when}\quad h \to +\infty. 
\end{split}
\end{equation}
More precisely, for $r,s \geq 0$ there exist constants $C_{r,s}$ such that
\begin{equation}
\label{prop73-4}
\begin{split}
&\left\| h^{r+1}\partial_h^{r}\partial_{\theta}^s ( x(\theta,h) - \epsilon_* \theta )\right\|_{L^{^\infty} ( ( -\pi,\pi) \times (2M_0,+\infty))} \leq C_{r,s},\\
&\left\|\sqrt{h} \Big[  h^{-\frac12 + r} \partial_h^{r} v(\theta,h)  - \epsilon_* \omega_r  \Big]\right\|_{L^{^\infty} ( ( -\pi,\pi) \times (2M_0,+\infty))} \leq C_{r},\quad\mbox{and}\\
\mbox{for}\quad s \geq 1, \quad &\left\|\sqrt{h} \Big[  h^{\frac12 + r} \partial_h^{r} \partial_{\theta}^s v(\theta,h)  \Big]\right\|_{L^{^\infty} ( ( -\pi,\pi) \times (2M_0,+\infty))} \leq C_{r,s}.
\end{split}
\end{equation}
Finally, we have for $r,s \geq 1,$
\begin{equation}
\label{prop73-5}
\begin{split}
&\Norm{ |(h -M_0)^{r}|\log(h - M_0)|^{-s+2}\partial_h^{r}\partial_{\theta}^s( x(\theta,h) - \epsilon_*\theta )}{L^{^\infty} ( ( -\pi,\pi) \times (M_0,2M_0)} \leq C_{r,s}\quad\mbox{and}\\
&\Norm{ |(h -M_0)^{r}|\log(h - M_0)|^{-s + 3}\partial_h^{r}\partial_{\theta}^s( v(\theta,h))}{L^{^\infty} ( ( -\pi,\pi) \times (M_0,2M_0)} \leq C_{r,s}
\end{split}
\end{equation}
for some constants $C_{r,s}$, and 
$$\Norm{|\log(h - M_0)|^{-s} \partial_\theta^s (x(\theta,h) - \epsilon_*\theta)}{L^{^\infty} ( ( -\pi,\pi) \times (M_0,2M_0)} \leq C_s$$
for $s \geq 1$ and some constant $C_s$. 
\end{proposition}
\begin{proof}
We begin with the study of the function $\omega_{*},$ and prove \eqref{prop73-1} and \eqref{prop73-2}. 
When $h \to +\infty,$ $k(h)$ goes to $+\infty$, and $1/k(h)$ goes to $0$. Hence as $\bK(z)$ extends near $z\sim0$ as a smooth function in $z^2$ with $\bK(0) = \frac{\pi}{2}$,  we have that 
$$
\omega_{*}(h) = \frac{\pi k(h) \sqrt{M_0}} {\bK\left(\frac{1}{k(h)}\right)}
= \sqrt{2(h + M_0)}\left(1 +  \Omega\left(\frac{1}{h+M_0}\right)\right). 
$$
where $\Omega$ is analytic in a neighborhood of $0$. This shows that on $(2M_0, +\infty)$, we have $\omega_{*}(h) = \sqrt{2h}( 1+ \tilde{\Omega}(1/h))$ for some analytic function $\tilde \Omega$ on $(0,\frac{1}{2M_0})$. 
This gives the first asymptotic of \eqref{prop73-1}, and also the first estimate of \eqref{prop73-2}.

When $h\to M_{0}^{+},$ $1/k(h)$ goes to $1$ and is smooth in a neighborhood of $M_0.$ Moreover $1 - \frac{1}{k(h)} \sim  \frac{1}{2M_0} (h - M_0)$ when $h \to M_0^+$. Asymptotics \eqref{rachm1} show that 
$$\omega_{*}(h) =  \frac{\pi k(h) \sqrt{M_0}} {\bK\left(\frac{1}{k(h)}\right)} \sim \pi \sqrt{M_0} \frac{2}{(- \log (1 - \frac{1}{k(h)}))} \quad \mbox{when} \quad h \to M_0^+,$$
from which we infer the asymptotics of $\omega_{*}$. The second estimate of \eqref{prop73-2} is easily deduced using the estimate on $1/\bK$ from \eqref{rachm}, and also the fact that $1/k(h)$ is smooth on $(M_{0},2M_{0}).$

Let us now study the functions $x(\theta,h)$ and $v(\theta,h).$ Using the expansions of \eqref{amuk}, and the expressions of $x(\theta,h)$ and $v(\theta,h)$ from Proposition \ref{prop72}, we write 
$$x(\theta,h) = \epsilon_* \theta  + \epsilon_*  4 \sum_{m = 0}^\infty \frac{q(1/k(h))^{m+1}}{(m+1) ( 1+ q(1/k(h))^{2(m+1)})}
\sin((m+1) \theta) $$
and
$$
v(\theta,h) = \epsilon_*  \omega_{*}(h) \left(  1 + 4  \sum_{m = 0}^\infty \frac{q(1/k(h))^{m+1}}{  1+ q(1/k(h))^{2(m+1)}}
\cos((m+1) \theta)\right) 
= \omega_{*}(h) \partial_\theta x(\theta,h). 
$$
Note that by construction, $x(\theta,h)$ is bounded for $(\theta,h) \in (-\pi,\pi) \times (M_0,+\infty)$. It is then clear that we have to consider and study the auxiliary function
$$R(\theta,q) =  4 \sum_{m = 0}^\infty \frac{q^{m+1}}{ 1+ q^{2(m+1)}} e^{i (m+1) \theta}.$$
This function is well defined for $|q| < 1$ and when $q \to 0$, we have 
\begin{equation}
\label{mji1}
R(\theta,q) \sim 4 q e^{i \theta}.
\end{equation}
Moreover, we have the estimate
\begin{equation}
\label{mji2}
\Norm{\partial_q^r \partial_\theta^s R(\theta,q)}{L^\infty(0,\frac12)}\leq C_{r,s}
\end{equation}
for all $r,s \geq 0$ and some constant $C_{r,s}$.To prove this, note that as $(1+x^{2})^{-1}$ and its derivatives are bounded near $x=0,$ we can write that
$$\partial_q^r \partial_\theta^s R(q,\theta) =  \sum_{m \geq  r -1}^\infty (m+1)^{r+s} q^{m+1 - r}e^{i (m+1) \theta}R_m^{r,s}(q)$$
where $R_m^{r,s}(q) \leq C_{r,s}$ for all $m$ and $q \in (0,\frac12).$ Estimate \eqref{mji2} follows easily.
In addition, when $q \to 1$, we have 
\begin{equation}
\label{kristine}
\Norm{R(\theta,q)}{L^{\infty}(\frac12,1)} \leq 4 \sum_{m = 0}^\infty q^{m+1}\leq 4 \left(\frac{q}{1 - q}\right).
\end{equation}
We also want to estimate the derivatives of this function. We can proceed as previously, and use the fact that $x \mapsto (1 + x^2)^{-1}$ is bounded as well as all its derivatives near $x = 1$, so that we can write for $r \geq 1$  and $s \geq 0$; 
$$\partial_q^r \partial_\theta^s R(\theta,q) =  \sum_{m \geq  r -1}^\infty (m+1)^{r+s} q^{m+1 - r}e^{i (m+1) \theta}R_m^{r,s}(q)$$
where $R_m^{r,s}(q) \leq C_{r,s}$ for all $m$ and $q \in (\frac12,1)$. We deduce that 
\begin{equation}
\label{schue}
\Norm{\partial_q^r \partial_\theta^s R(\theta,q)}{L^{\infty}(\frac12,1)} \leq \frac{C_{r,s}}{(1 - q)^{r+s +1}}, 
\end{equation}
for some constant $C_{r,s}$. 

We can now study the asymptotic behaviors of $x(\theta,h)$ and $v(\theta,h),$ starting with the case when $h \to +\infty.$ In this case $1/k(h)$ goes to $0,$ and $q(1/k(h))$ is a smooth function of $1/k(h)^2 = 2M_0/(h+M_0) \to 0.$ As we have
$$x(\theta,a)=\eps_{*}\left[\theta + \mathrm{Im}\left(R\left(\theta,q\left(\frac{1}{k(h)}\right)\right)\right)\right],$$ 
we deduce from then \eqref{mji1}, \eqref{mji2} and the Fa\`a di Bruno formula that 
$$\partial_h^r \partial_\theta^s  ( x(\theta,h) - \epsilon_* \theta)  = \mathcal{O}(h^{-r - 1}) $$
uniformly in $\theta$. The results for $v$ are obtained from the previous result and the fact that $v(h,\theta) = \omega_{*}(h) \partial_\theta x(\theta,h) \sim \epsilon_*  \omega_{*}(h) \sim \epsilon_* \sqrt{2h}$ when $h \to 0$. In other words asymptotics \eqref{prop73-3} and estimates \eqref{prop73-4} are proved.

It remains to study the behaviors of $x(\theta,h)$ and $v(\theta,h)$ when $h \to M_0^+.$  Using the properties of $1/(1- q(z))$ (see \eqref{rachm1} and \eqref{rachm}) near $z = 1$, we obtain from \eqref{schue}
$$\Norm{\partial_q^r \partial_\theta^s R(q(1/k(h)),\theta)}{L^{\infty}(M_0,2M_0)} \leq C \log\left( \frac{1}{h - M_0}\right) ^{r + s +1}, $$
and moreover, from \eqref{rachm}
$$\partial_h^r q(1/k(h)) = \mathcal{O}\left(\left(\frac{1}{h-M_0}\right)^{r} \frac{1}{\log (h - M_0)^{2}}\right). $$
This shows that for $s \geq 1$, 
$$\partial_\theta^s ( x(\theta,h) - \epsilon_* \theta) = \mathcal{O} \left(\log\left( \frac{1}{h - M_0}\right) ^{s+1}\right), $$
and using the Fa\`a di Bruno formula, we see that for $s \geq 1$ and $r \geq 1$, 
$$\partial_h^r \partial_{\theta}^s (x(\theta,h) - \epsilon_* \theta) = \mathcal{O} \left( \left(\frac{1}{h-M_0}\right)^{r} \log \left(\frac{1}{h - M_0}\right)^{s-1}\right).  $$
As $v(\theta,h) = \omega_{*}(h) \partial_\theta x(\theta,h)$, we deduce from the estimates on $\omega_{*}(h)$ that 
$$\partial_h^r \partial_{\theta}^s v(\theta,h)  = \mathcal{O} \left( \left(\frac{1}{h-M_0}\right)^{r} \log \left(\frac{1}{h - M_0}\right)^{s-3}\right), $$
and this proves estimates \eqref{prop73-5}. To conclude the proof, let us say that the analyticity of $\omega_{*}(h),$ $x(\theta,h)$ and $v(\theta,h)$ stated in the Proposition follows from the analyticity properties of the special functions stated at the beginning of section \ref{actionangle}.
\end{proof}

Now let us consider a function $f(x,v)$ that is continuous and its restriction $f^*$ to $U_*$. We are interested in the behavior of the Fourier coefficients \eqref{coefou} that by a slight abuse of notation, we will also denote by
$$f_\ell^*(h) = \frac{1}{2\pi}\int_{0}^{2\pi} f^*(x(\theta,h), v(\theta,h))e^{- i \ell \theta} \dd \theta $$
these coefficients in the variable $h$.\\ In the special cases where f is either the cosine or the sine function, we have the following explicit expressions:

\begin{proposition}
\label{prop74}
For $*\in\{\pm\},$ and $(\theta,a) \in J_{*}\times (-\pi,\pi),$
$$\cos(x(\theta,a)) = \sum_{\ell \in \Z} C_\ell^*(a) e^{i\ell \theta}\quad \mbox{and} \quad \sin(x(\theta,a)) = \sum_{\ell \in \Z} S_\ell^*(a) e^{i\ell \theta}. $$
with, in terms of the variable $h$, 
\begin{equation}
\label{coeffouriercos1}
\begin{split}
&C^*_0(h) = 1 - 2 k(h)^2  + 2k(h)^2 \frac{\bE\left(\frac1{k(h)}\right)}{ \bK\left(\frac1{k(h)}\right)} \\
&C^*_\ell(h) = C^*_{-\ell}(h)=\frac{2\pi^2 k(h)^2 }{ \bK\left(\frac1{k(h)}\right)^2} \left( \frac{ |\ell| q\left(\frac1{k(h)}\right)^{|\ell|}}{1 - q\left(\frac1{k(h)}\right)^{2|\ell|}} \right), \quad \ell >0
\end{split}
\end{equation}
and
\begin{equation}
\label{coeffouriersin1}
\begin{split}
&S^*_0(h) = 0 \\
&S^*_\ell(h) =-S^*_{-\ell}(h)= \epsilon_* (-i) \frac{2\pi^2 k(h)^2 }{ \bK\left(\frac1{k(h)}\right)^2} \left( \frac{ |\ell| q\left(\frac1{k(h)}\right)^{|\ell|}}{1 + q\left(\frac1{k(h)}\right)^{2|\ell|}} \right), \quad \ell >0.
\end{split}
\end{equation}
\end{proposition}

\begin{proof}
Recall that we have
$$\sin(x(\theta,h)/2) =\epsilon_*\sn\left(\frac{1}{\pi}\bK\left(\frac1{k(h)}\right) \theta , \frac{1}{k(h)}\right). $$ 
Hence using the expansion of $\sn^{2}$ in \eqref{Milne} and the second formula of \eqref{classictrigo}, we obtain 
\begin{multline*}
\cos(x(\theta,h)) = 1 - 2\sn^2\left(\frac{1}{\pi}\bK\left(\frac1{k(h)}\right) \theta , \frac{1}{k(h)}\right) \\ 
= 1 - 2 k(h)^2  \frac{\bK\left(\frac1{k(h)}\right) - \bE\left(\frac1{k(h)}\right)}{ \bK\left(\frac1{k(h)}\right)} 
+
  \frac{4\pi^2 k(h)^2 }{ \bK\left(\frac1{k(h)}\right)^2} \sum_{m = 1}^\infty \frac{m q\left(\frac1{k(h)}\right)^{m}}{1 - q\left(\frac1{k(h)}\right)^{2m}} \cos( m \theta). 
\end{multline*}
Formulae \eqref{coeffouriercos1} follow easily. Moreover, we have 
$$x(\theta,h) = \epsilon_* 2\, \am \left(\frac{1}{\pi}\bK\left(\frac1{k(h)}\right) \theta , \frac{1}{k(h)}\right). $$
Hence, using the third formula of \eqref{classictrigo}, the definitions of the functions $\sn(u,k)$ and $\cn(u,k),$ and the expansion of $\sn(u,k)\cn(u,k)$ from \eqref{Milne}, we infer 
\begin{eqnarray*}
\sin(x(\theta,h)) &=& 2\epsilon_* \sn\left(\frac{1}{\pi}\bK\left(\frac1{k(h)}\right) \theta , \frac{1}{k(h)}\right)\cn\left(\frac{1}{\pi}\bK\left(\frac1{k(h)}\right) \theta , \frac{1}{k(h)}\right)\\
&=& 2\epsilon_* \frac{2 k(h)^2 \pi^2}{\bK(1/k(h))^2} \sum_{m = 1}^\infty \frac{m q(1/k(h))^{m}}{1 + q(1/k(h))^{2m}} \sin( m \theta). 
\end{eqnarray*}
Formulae \eqref{coeffouriersin1} follow easily.
\end{proof}

For some smooth function $f$, we can estimate the generalized Fourier coefficients $f_{\ell}^*(a)$ in the following way.

\begin{proposition}
\label{prop75}
Assume that $f$ is a function satisfying 
$$\max_{|\alpha| \leq m} \Norm{\langle v \rangle^{\mu}\partial_{x,v}^{\alpha} f}{L^{\infty}(U_*)}\leq C_{p,m}$$
for some $m\geq 1$ and $\mu \geq 0$.  Then, we have for $r+s \leq m$.
\begin{equation}
\begin{split}
&\partial_h^r f_\ell^*(h) = \mathcal{O} \left( \frac{1}{|\ell|^s }\Big(\frac{1}{h-M_0}\Big)^{r} \log \Big(\frac{1}{h - M_0}\Big)^{s}\right) \quad\mbox{when}\quad h \to M_0^+\\
&\partial_h^r f_\ell^*(h) = \mathcal{O}\left(\frac{1}{|\ell|^s} \frac{1}{h^{\mu/2}} \right) \quad\mbox{when}\quad h \to +\infty. 
\end{split}
\end{equation}
\end{proposition}

\begin{proof} 
We have 
$$\partial_h^r f_\ell^*(h) = \frac{1}{\ell^s}(-i)^s \frac{1}{2\pi}\int_{(-\pi,\pi)} \partial_h^r \partial_{\theta}^s ( f^*(x(\theta,h), v(\theta,h))) e^{- i \ell \theta} \dd \theta,$$
and using Fa\`a di Bruno formula, the hypothesis on $f,$ and \eqref{prop73-5}, we infer that when $h \to M_0^+,$
$$\partial_h^r f_\ell^*(h) = \mathcal{O} \left( \frac{1}{|\ell|^s }\Big(\frac{1}{h-M_0}\Big)^{r} \log \Big(\frac{1}{h - M_0}\Big)^{s}\right) .$$
Moreover, by the same arguments, and using also \eqref{prop73-4}, we obtain  
$$\partial_h^r f_\ell^*(h) = \mathcal{O}\Big(\frac{1}{|\ell|^s} \frac{1}{h^{p/2}} \Big), \qquad h \to +\infty.$$
\end{proof}

From this result, we obtain the following: 

\begin{proposition}
\label{decaydehors}
Assume that $f$ an $\varphi$ are real functions satisfying 
$$\max_{|\alpha | \leq m} \Norm{\langle v \rangle^{\mu}\partial_{x,v}^\alpha f}{L^{\infty}(U_*)}\leq C_{\mu,m}
\quad\mbox{and}\quad
\max_{|\alpha | \leq M} \Norm{\partial_{x,v}^\alpha  \varphi}{L^{\infty}(U_*)}\leq C_{M}. $$
for some $m \geq 1 + r$,  $\mu > 2$ and $M \geq 3+r$. 
Then we have for all $t\geq 1,$
$$\left|\int_{U_*} \varphi(\psi_t(x,v)) f(x,v) \dd v -  \int_{M_0}^{+\infty} f_0^*(h) \varphi_0^*(h) \frac{1}{\omega_{*}(h)} \dd h \right|  \leq \frac{C}{(1 + t)^{r+1}} .$$
\end{proposition}
%
%In this statement, $\varphi$ has to be thought as a test function so that we can think $M$ to be as big as we need. Note that $\varphi$
%is not assumed to be localized in $v$ since we shall mainly use this result for test  functions that do not depend on $v$
%like $\cos(x)$ or $\sin(x).$

\begin{proof}
Recall that for $*\in\{\pm\},$ we have by formulae \eqref{oslo} and \eqref{trondheim}, and the identity $\psi_{t}(\theta,h)=\theta + t \omega_{*}(h),$
$$\int_{U_*} \varphi(\psi_t(x,v)) f(x,v) \dd v \dd x  = \sum_{\ell\in \Z} \int_{(M_0,+\infty)}f_\ell^*(h) \varphi_{-\ell}^*(h)  e^{i t \ell \omega_{*}(h) } \frac{1}{\omega_{*}(h)} \dd h.$$
We shall then integrate by parts with respect to $h,$ using in particular the fact that $\partial_{h} \omega_{*}$ does not vanish on $U_{\pm}.$ For $\ell\neq 0,$ we may thus define two operators $\mathrm{D}_{\ell,h}$ and $\mathrm{D}_{\ell,h}^{\top}$ acting on function $G(h)$ of $h$ by
$$\mathrm{D}_{\ell,h}G=\frac{1}{i\ell \partial_{h}\omega_{*}} \partial_{h}G \quad \mbox{and} \quad \mathrm{D}_{\ell,h}^{\top}G =-\partial_{h}\left(\frac{G}{i\ell\partial_{h}\omega_{*}}\right).$$
We shall in particular consider iterations of the operator $\mathrm{D}_{\ell,h}^{\top},$ and use the special notation
$$(\mathrm{D}_{\ell,h}^{\top})^{0}G =-\frac{G}{i\ell\partial_{h}\omega_{*}}.$$
We have the following useful Lemma (whose proof is postponed to the end of the current one):
\begin{lemma}
\label{lemrec}
For all $0\leq \beta \leq r+1,$
$$(\mathrm{D}_{\ell,h}^{\top})^{\beta}\left( f_{\ell}^{*}(h)\varphi_{-\ell}^{*}(h) \omega_{*}(h)^{-1} \right) =\mathcal{O}(|\ell|^{-\beta}h^{-\mu/2}) \quad \mbox{when} \quad h\to +\infty,$$
and
$$(\mathrm{D}_{\ell,h}^{\top})^{\beta}\left( f_{\ell}^{*}(h)\varphi_{-\ell}^{*}(h) \omega_{*}(h)^{-1} \right) =\mathcal{O}\left(|\ell|^{-\beta}(h-M_{0})\log\left(\frac{1}{h-M_{0}}\right)^{m(\beta)}\right) \quad \mbox{when} \quad h\to M_{0}^{+},$$
for some integer $m(\beta)>0.$
\end{lemma}
This Lemma says essentially that the singularities coming from $f_{\ell}^*$ and $\varphi_{-\ell}^{*}$ at the separatix are cancelled by the one of $\partial_{h}\omega_{*}.$ This shows in particular that for all $0< \beta\leq r+1,$
$$\lim_{h\to M_{0}^{+}} (\mathrm{D}_{\ell,h}^{\top})^{\beta}\left( f_{\ell}^{*}(h)\varphi_{-\ell}^{*}(h) \omega_{*}(h)^{-1} \right) \frac{1}{\partial_{h}\omega_{*}(h)}=0,$$
$$\lim_{h\to +\infty} (\mathrm{D}_{\ell,h}^{\top})^{\beta}\left( f_{\ell}^{*}(h)\varphi_{-\ell}^{*}(h) \omega_{*}(h)^{-1} \right) \frac{1}{\partial_{h}\omega_{*}(h)}=0,$$
and that
$$\left\|(\mathrm{D}_{\ell,h}^{\top})^{\beta}\left( f_{\ell}^{*}(h)\varphi_{-\ell}^{*}(h) \omega_{*}(h)^{-1} \right)\right\|_{L^{1}((M_{0},+\infty))} \leq \frac{C_{\beta}}{|\ell|^{\beta}},$$
using also \eqref{prop73-2} and the hypothesis $\mu > 2.$\\
Integrating by parts $r+1$ times while using Lemma \ref{lemrec}, we have for $\ell \neq 0$
\begin{equation}
\notag
\begin{split}
\int_{(M_{0},+\infty)}f_{\ell}^{*}(h)\varphi_{-\ell}^{*}(h)e^{i\ell t \omega_{*}(h)} \frac{\dd h}{\omega_{*}(h)} & =\frac{1}{t}\int_{(M_{0},+\infty)}f_{\ell}^{*}(h)\varphi_{-\ell}^{*}(h)\mathrm{D}_{\ell,h}(e^{i\ell t \omega_{*}(h)}) \frac{\dd h}{\omega_{*}(h)}\\
&=\frac{1}{t}\left[ -(\mathrm{D}_{\ell,h}^{\top})^{0}\left( f_{\ell}^{*}(h)\varphi_{-\ell}^{*}(h) \omega_{*}(h)^{-1} \right) e^{i\ell t\omega_{*}(h)}\right]_{M_{0}}^{+\infty}\\
&+ \frac{1}{t}\int_{(M_{0},+\infty)} e^{i\ell t \omega_{*}(h)} (\mathrm{D}_{\ell,h}^{\top})\left( f_{\ell}^{*}(h)\varphi_{-\ell}^{*}(h)\omega_{*}(h)^{-1} \right) \dd h,\\
&=\frac{1}{t}\int_{(M_{0},+\infty)} \mathrm{D}_{\ell,h}(e^{i\ell t \omega_{*}(h)} )(\mathrm{D}_{\ell,h}^{\top})\left( f_{\ell}^{*}(h)\varphi_{-\ell}^{*}(h)\omega_{*}(h)^{-1} \right) \dd h,\\
&= \frac{1}{t^{r+1}}\int_{(M_{0},+\infty)} e^{i\ell t \omega_{*}(h)} (\mathrm{D}_{\ell,h}^{\top})^{r+1}\left( f_{\ell}^{*}(h)\varphi_{-\ell}^{*}(h)\omega_{*}(h)^{-1} \right) \dd h,
\end{split}
\end{equation}
such that
\begin{multline*}\left| \int_{(M_{0},+\infty)}f_{\ell}^{*}(h)\varphi_{-\ell}^{*}(h)e^{i\ell t \omega_{*}(h)} \frac{\dd h}{\omega_{*}(h)} \right|\leq \frac{1}{t^{r+1}}\left\| (\mathrm{D}_{\ell,h}^{\top})^{r+1}\left( f_{\ell}^{*}(h)\varphi_{-\ell}^{*}(h)\omega_{*}(h)^{-1} \right) \right\|_{L^{1}((M_{0},+\infty))} \\ \leq \frac{C_{r+1}}{t^{r+1}|\ell|^{r+1}},
\end{multline*}
and summing in $\ell$ gives the result.
\end{proof}

\begin{Proofof}{[Lemma \ref{lemrec}]}
Let us first prove the estimate when $h\sim M_{0}.$ We shall rather prove by induction on $\beta\leq r+1$ that
\begin{equation}
\label{1}
(\mathrm{D}_{\ell,h}^{\top})^{\beta}\left( f_{\ell}^{*}(h)\varphi_{-\ell}^{*}(h) \omega_{*}(h)^{-1} \right) =\mathcal{O}\left(|\ell|^{-\beta}(h-M_{0})\log\left(\frac{1}{h-M_{0}}\right)^{m(\beta)}\right),
\end{equation}
and, if $\beta\leq r,$
\begin{equation}
\label{2}
\partial_{h}(\mathrm{D}_{\ell,h}^{\top})^{\beta}\left( f_{\ell}^{*}(h)\varphi_{-\ell}^{*}(h) \omega_{*}(h)^{-1} \right) =\mathcal{O}\left(|\ell|^{-\beta}\log\left(\frac{1}{h-M_{0}}\right)^{n(\beta)}\right),
\end{equation}
for some integers $m(\beta),n(\beta).$\\
When $\beta=0,$ we have by Proposition \ref{prop75}, \eqref{prop73-1} and \eqref{prop73-2}
$$-(\mathrm{D}_{\ell,h}^{\top})^{0}\left( f_{\ell}^{*}(h)\varphi_{-\ell}^{*}(h) \omega_{*}(h)^{-1}\right)=\frac{f_{\ell}^{*}(h)\varphi_{-\ell}^{*}(h) }{\omega_{*}(h)\partial_{h}\omega_{*}(h)}=\mathcal{O}\left((h-M_{0})\log\left(\frac{1}{(h-M_{0})}\right)^{3}\right),$$
and
\begin{equation}
\notag
\begin{split}
\partial_{h}\left(\frac{f_{\ell}^{*}(h)\varphi_{-\ell}^{*}(h)}{ \omega_{*}(h)\partial_{h}\omega_{*}(h)}\right)&=\frac{\partial_{h}\left(f_{\ell}^{*}(h)\varphi_{-\ell}^{*}(h)\right)}{\omega_{*}(h)\partial_{h}\omega_{*}(h)} +  \left(f_{\ell}^{*}(h)\varphi_{-\ell}^{*}(h)\right) \partial_{h}\left(\frac{1}{\omega_{*}(h)\partial_{h}\omega_{*}(h)}\right)\\
&= \mathcal{O}\left(\log\left(\frac{1}{(h-M_{0})}\right)^{3}\right).
\end{split}
\end{equation}
For $\beta\geq 1,$ if \eqref{1} and \eqref{2} hold at rank $\beta-1,$ then by using the formula
$$(\mathrm{D}_{\ell,h}^{\top})^{\beta}\left( f_{\ell}^{*}(h)\varphi_{-\ell}^{*}(h) \omega_{*}(h)^{-1} \right) = -\partial_{h}\left(\frac{(\mathrm{D}_{\ell,h}^{\top})^{\beta-1}\left( f_{\ell}^{*}(h)\varphi_{-\ell}^{*}(h) \omega_{*}(h)^{-1} \right)}{i\ell \partial_{h}\omega_{*}(h)}\right),$$
and the estimate \eqref{prop73-2} for $\partial_{h}\omega_{*}(h)^{-1},$
one easily proves \eqref{1} at rank $\beta.$ As long as $\beta\leq r,$ one deduces then \eqref{2} at rank $\beta$ by writing that 
$$\log\left(\frac{1}{h-M_{0}}\right)^{-m(\beta)}(\mathrm{D}_{\ell,h}^{\top})^{\beta}\left( f_{\ell}^{*}(h)\varphi_{-\ell}^{*}(h) \omega_{*}(h)^{-1} \right) =\mathcal{O}\left(|\ell|^{-\beta}(h-M_{0})\right),$$
which shows that
$$\partial_{h}\left(\log\left(\frac{1}{h-M_{0}}\right)^{-m(\beta)}(\mathrm{D}_{\ell,h}^{\top})^{\beta}\left( f_{\ell}^{*}(h)\varphi_{-\ell}^{*}(h) \omega_{*}(h)^{-1} \right)\right)=\mathcal{O}(|\ell|^{-\beta}),$$
and gives the result.\\
The asymptotics when $h\to +\infty$ are easier to obtain, as for any $s\leq \beta\leq r+1,$ Proposition \ref{prop75} implies that
$$\partial_{h}^{s}\left(f_{\ell}^{*}(h)\varphi_{-\ell}^{*}(h)\right)=\mathcal{O}(h^{-\mu/2}),$$
while \eqref{prop73-2} shows that
$$\partial_{h}^{s}\left(\frac{1}{\omega_{*}(h)\partial_{h}\omega_{*}(h)} \right)= \mathcal{O}(1),$$
such that Leibniz's formula yields the result. Note that the contribution $|\ell|^{-\beta}$ comes obviously from iterations of the operator $\mathrm{D}_{\ell,h}^{\top}.$
\end{Proofof}

\subsection{Action-angle in $U_{\circ}$}

In this subsection we provide a rather complete description of the change of variable in $U_{\circ}.$

\begin{proposition}
\label{prop78}
There exists a symplectic change of variable $(x,v) \mapsto (\psi,h)$ from $U_{\circ}$ to the set 
$$V_\circ := \{(\psi,h) \in \R^2, | h \in (-M_0,M_0), \, \psi \in (-r_{\circ}(h),r_{\circ}(h))\}, $$
with $$
r_{\circ}(h) = \frac{2 k(h)}{ \sqrt{M_0}} \bK(k(h)), \quad \mbox{where} \quad k(h) = \sqrt{\frac{h+M_0}{2M_0}}. $$
such that the flow of the pendulum in the variable $(\psi,h)$ is $h(t) = h(0)$ and $\psi(t) = t + \psi(0)$. 
There exists then a second symplectic change of variables $(\psi,h) \mapsto (\theta,a)$ from $V_{\circ}$ to 
$$\{(\theta,a) \in \R^2, | a \in J_{\circ}=\left(0,\frac{8}{\pi}\sqrt{M_0}\right), \, \theta \in (-\pi,\pi)\}, $$
such that 
$$\left\{
\begin{array}{lcl}
a(h) &=& \displaystyle  \frac{8 \sqrt{M_0}}{\pi} ( \bE(k(h)) - (1 - k(h) ^2) \bK(k(h)) ) \\[2ex]
\theta(\psi,h) &=& \omega_{\circ}(h) \psi
\end{array}\right.
\quad\mbox{with}\quad 
\omega_{\circ}(h) = \frac{\pi\sqrt{M_0}  }{2  \bK(k(h))}$$
and so that the flow of the pendulum in the variables $(\theta,a)$ is $a(t) = a(0)$ and $\theta(t) = t \omega_{\circ}(a(0)) + \psi(0).$\\
Moreover, we can easily express $(x,v)$ as functions of the variables $(\theta,h)$ with the formulae
\begin{eqnarray}
x(\theta,h) &=&  2 \arcsin \left( k(h) \sn\left(\frac{2}{\pi} \bK(k(h)) \left(\theta + \frac{\pi}{2}\right),k(h)\right) \right) , \label{220} \\
v(\theta,h) &=&    2k(h) \sqrt{M_0}\, \cn \left(\frac{2}{\pi} \bK(k(h)) \left(\theta + \frac{\pi}{2}\right),k(h)\right). \label{221}
\end{eqnarray}
\end{proposition}
\begin{remark}
\label{rem79}
Note that we can check directly from the formulae that $\omega_\circ(h)$ is decreasing, and as $h$ is a strictly increasing function of $h$ (see\eqref{omh}),  $\omega_\circ(a)$ is decreasing, and $\partial_a \omega_\circ(a) <0$, $a \in J_\circ$. 
\end{remark}
\begin{proof}
In this case, we have $h \in (-M_0,M_0)$ and we can write 
$$v (x,h) = \epsilon_*\sqrt{2(h + M_0 \cos (x))}$$
defined for $h + M_0 \cos(x) \geq 0$, where $\epsilon_* = 1$ if $v \geq 0$ and $\epsilon_* = -1$ if $v \leq 0$. 
Using this representation, both sets $U_{\circ,+} = U_\circ \cap \{ v \geq 0\}$ and $U_{\circ,-} = U_\circ \cap \{ v \leq 0\}$ can be parametrized as  
\begin{eqnarray*}
U_{\circ,*} &=& \{ (x,h) \in \T \times (-M_0,M_0) \, | \, h \geq -M_0 \cos(x)\}\\
&=& \{ (x,h) \, | \, h \times (-M_0,M_0), x \in ( -x_0(h),x_0(h))\},
\end{eqnarray*}
where $x_0(h)$ is the solution in $[0,\pi]$ of the equation $h + M_0 \cos(x_0(h)) = 0$.
Note that have 
$$\sin^2(x_0(h)/2)  = k(h)^2.$$ 
For $x \in (-x_0(h),x_0(h))$, 
let us define $\Theta(x,h) \in (-\pi/2,\pi/2)$ as the unique solution of 
$$k(h) \sin (\Theta(x,h)) = \sin(x/2). $$
This solution is well defined when $x \in (-x_0(h),x_0(h))$ as $\frac{1}{k(h)}\sin(\frac{x}2) \in (0,1)$ in this interval.
Note that $\Theta(0,h) = 0$, $\Theta(-x_0(h),h) = -\frac{\pi}{2}$ and $\Theta(x_0(h),h) = \frac{\pi}{2}$. Moreover, by taking the derivative with respect to $x$, we have 
$$
k(h) \cos( \Theta(x,h)) \partial_x \Theta(x,h) = \frac12\cos(x/2)=\frac{1}{2} \sqrt{1 - k(h)^2 \sin(\Theta(x,h))^2}. 
$$
In particular, we have 
$$\sqrt{1 - \sin^2(\Theta(x,h))} = \frac{1}{2 k(h)  \partial_x \Theta(x,h)} \sqrt{1 - k(h)^2 \sin(\Theta(x,h))^2}. $$
Then we have 
$$U_{\circ,*} = \{ (x,h) |  h \in (-M_0,M_0), \Theta(x,h) \in (-\pi/2,\pi/2)\}, $$
and we can define the generatrix function $S(x,h)$ on $U_{\circ,*}$ by the formula
\begin{eqnarray*}
S(x,h) &=& \epsilon_* \int_{x_0(h)}^x \sqrt{2(h + M_0 \cos(y))} \dd y\\
&=&  \epsilon_* \int_{x_0(h)}^{x} \sqrt{2(h + M_0) - 4M_0 \sin^2(y/2)} \dd y\\
&=& \epsilon_* 2k(h)\sqrt{M_0} \int_{x_0(h)}^{x} \sqrt{1 - \frac{\sin^2(y/2)}{k(h)^2}} \dd y \\
&=& \epsilon_* 2 k(h) \sqrt{M_0} \int_{x_0(h)}^{x} \sqrt{1 - \sin^2 (\Theta(y,h)}) \dd y\\
&=& \epsilon_* 2 k(h) \sqrt{M_0}  \int_{x_0(h)}^{x} \sqrt{1 -  \sin^2 (\Theta(y,h)}) \frac{1}{\partial_x\Theta(y,h) } \partial_x \Theta(y,h) \dd y \\
&=&\epsilon_* 4 k(h)^2 \sqrt{M_0}  \int_\frac{\pi}{2}^{\Theta(x,h)} \frac{1 -  \sin^2 (\phi)}{\sqrt{1 - k(h)^2 \sin^2 (\phi) }} \dd \phi. 
\end{eqnarray*}
Whence 
\begin{multline*}
S(x,h) = \epsilon_* 4\sqrt{M_0}  ( E(\Theta(x,h),k(h)) - \bE(k(h)))  \\ - \epsilon_* 4\sqrt{M_0} (1 - k(h) ^2) ( F(\Theta(x,h),k(h)) -  \bK(k(h)) ). 
\end{multline*}
Note that this function is equal to zero on the axis $\{ v = 0, x \in [0,\pi]\}$ and has a discontinuity in the axis $\{ v = 0, x \in [-\pi,0]\}$. 
We can then define 
\begin{eqnarray*}
\psi(x,h) &=& \frac{\partial}{\partial h } S(x,h)\\
&=& \epsilon_*  \int_{x_0(h)}^x \frac{1}{\sqrt{2(h + M_0 \cos(y))}} \dd y\\
&=& \epsilon_*  \frac{1}{\sqrt{M_0}}  \int_{x_0(h)}^{x} \frac{1}{\sqrt{1 - k(h)^2 \sin^2 (\Theta(y,h)}) }\partial_x \Theta(x,h) \dd y\\
&=&  \epsilon_*  \frac{1}{\sqrt{M_0}} ( F(\Theta(x,k(h)),k(h)) - \bK(k(h) ) ), 
\end{eqnarray*}
where we used the fact that $h + M_0 \cos( x_0(h)) = 0.$\\
On a period, we thus see that $\psi(x,h) \in \left(-2\frac{\bK(k(h))}{\sqrt{M_0}},  2 \frac{\bK(k(h))}{\sqrt{M_0}} \right).$ 
Hence the function 
$$\theta = \epsilon_* \frac{\pi\sqrt{M_0}  }{2  \bK(k(h))}\psi =  \epsilon_* \frac{\pi}{2} \frac{F(\Theta(x,k(h)),k(h))}{\bK(k(h))}  - \epsilon_* \frac{\pi}{2}$$
belongs to $(-\pi,\pi),$ and is such that the point $x_0(h)$ correspond to the angle $\theta = 0$. 
The frequency and action are then given by 
$$\omega_{\circ} (h) = \frac{\pi\sqrt{M_0}  }{2  \bK(k(h))}  = \frac{1}{g_{\circ}(h)}\quad\mbox{and}\quad  
a_{\circ}(h) =  \frac{8 \sqrt{M_0}}{\pi} ( \bE(k(h)) - (1 - k(h) ^2) \bK(k(h)), $$
as 
\begin{eqnarray}
\partial_h a_{\circ}(h) &=& \frac{8 \sqrt{M_0}}{\pi}  \frac{1}{4 M_0 k(h) } (  \partial_k \bE(k(h)) - (1 - k(h)^2) \partial_k \bK(k(h) + 2 k(h) \bK(k(h)) \nonumber\\
&=& \frac{8 \sqrt{M_0}}{\pi}  \frac{1}{4 M_0 k(h) } ( k(h) \bK(k(h)) = g_{\circ}(h) >0.\label{omh} 
\end{eqnarray}
Using the properties of the elliptic functions, we have 
$$\Theta(x(\theta,h) ,h)) = \epsilon_* \am\left(\frac{2}{\pi} \bK(k(h)) \left(\theta+ \epsilon_* \frac{\pi}{2}\right),k(h)\right) $$
hence
\begin{equation}
\label{provisoire}
\sin(x(\theta,h)/2) = \epsilon_*  k(h) \sn\left(\frac{2}{\pi} \bK(k(h)) \left(\theta+ \epsilon_* \frac{\pi}{2}\right),k(h)\right) .
\end{equation}
Now using the first formula of \eqref{trigo2} we see that the expression \eqref{provisoire} does actually not depend on the value of $\epsilon^*=\pm1,$ and thus
$$\sin(x(\theta,h)/2) =   k(h) \sn\left(\frac{2}{\pi} \bK(k(h)) \left(\theta+  \frac{\pi}{2}\right),k(h)\right) ,$$
which yields 
$$x(\theta,h) =  2 \arcsin \left( k(h) \sn\left(\frac{2}{\pi} \bK(k(h)) \left(\theta+  \frac{\pi}{2}\right),k(h)\right) \right) . $$
Moreover, we have 
\begin{multline*}
v(\theta,h) =  \epsilon_* \sqrt{2(h + M_0 \cos (x(\theta,h))} 
= \epsilon_* \sqrt{2(h+M_0)} \sqrt{1 - \frac{1}{k(h)^2} \sin^2(x(\theta,h)/2)}  \\
= \epsilon_* \sqrt{2(h+M_0)} \sqrt{1 - \sin^2(\Theta(x(\theta,h),h)}
 =\epsilon_*  2k(h) \sqrt{M_0}\, \cn \left(\frac{2}{\pi} \bK(k(h)) \left(\theta+ \epsilon_* \frac{\pi}{2}\right),k(h)\right),
\end{multline*}
and using the second formula of \eqref{trigo2}, it yields 
$$v(\theta,h)= 2k(h) \sqrt{M_0}\, \cn \left(\frac{2}{\pi} \bK(k(h)) \left(\theta+  \frac{\pi}{2}\right),k(h)\right),$$
and concludes the proof.
\end{proof}

\begin{proposition}
\label{prop79}
The function $\omega_{\circ}(h)$, $x(\theta,h)$ and $v(\theta,h)$ are analytic for $\theta \in (-\pi,\pi)$ and $h \in (-M_0,M_0).$\\
The function $\omega_{\circ}$ exhibits the following asymptotic behavior
\begin{equation}
\label{prop79-1}
\begin{split}
&\omega_{\circ}(h) \sim \frac{\pi\sqrt{M_0}  }{ (- \log( M_0 - h)) }\quad \mbox{when}\quad h \to M_0^-,\\
&\omega_{\circ}(h) =  \sqrt{M_0} - \frac{1}{8\sqrt{M_0}} (h + M_0) \quad \mbox{when} \quad h \to -M_0,
\end{split}
\end{equation}
and there exists constants, $C_r$, $\omega_r$ with $\omega_1 = \frac{1}{8\sqrt{M_0}}$ and $\alpha_r \neq 0$ such that for all $r \geq 1$, 
\begin{equation}
\label{prop79-2}
\begin{split}
&\left\| (h+M_0)^{-1}  \Big[ \partial_h^r \omega_{\circ}(h) - \omega_r\Big]\right\|_{L^\infty(-M_0,0)} \leq C_r, \\ 
&\left\| \log(h-M_{0})^{2}\left[\log(h-M_{0})^{2}(h - M_0)^{ r} \partial_h^r \omega_{\circ}(h) - \alpha_r\right]\right\|_{L^\infty(0,M_0)} \leq C_r. 
\end{split}
\end{equation}
The change of variable $(h,\theta) \mapsto (x,v)$ satisfies the following estimates: When  $h \to -M_0$ it converges towards the action-angle variable of the harmonic oscillator $\frac{1}{2} ( v^2 + M_0 x^2)$ in the sense that
\begin{equation}
\label{prop79-3}
\begin{split}
&x(\theta,h) \sim  2\sqrt{\frac{h+M_0}{2M_0}}\cos(\theta) \quad \mbox{when}\quad h \to -M_0,\\
&v(\theta,h) \sim  - 2 \sqrt{M_0}\sqrt{\frac{h+M_0}{2M_0}}\sin(\theta) \quad \mbox{when}\quad h \to -M_0. 
\end{split}
\end{equation}
More precisely, for $r,s \geq 0$ there exist constants $C_{r,s}$ such that
\begin{equation}
\label{prop79-4}
\begin{split}
&\left\| (h+M_0)^{r-\frac{1}{2}}\partial_h^{r}\partial_{\theta}^s \Big[ x(\theta,h) - 2\sqrt{\frac{h+M_0}{2M_0}}\cos(\theta) \Big]\right\|_{L^{^\infty} ( ( -\pi,\pi) \times (-M_0,0))} \leq C_{r,s},\\
&\left\| (h+M_0)^{r-\frac{1}{2}}\partial_h^{r}\partial_{\theta}^s \Big[ v(\theta,h) +  2 \sqrt{M_0}\sqrt{\frac{h+M_0}{2M_0}}\sin(\theta) \Big]\right\|_{L^{^\infty} ( ( -\pi,\pi) \times (-M_0,0))} \leq C_{r,s}. 
\end{split}
\end{equation}
Finally, we have for $r,s \geq 1,$ 
\begin{equation}
\label{prop79-5}
\begin{split}
&\Norm{ |(M_0 -h)^{r}|\log(M_0 -h)|^{-s+2}\partial_h^{r}\partial_{\theta}^s( x(\theta,h)  )}{L^{^\infty} ( ( -\pi,\pi) \times (0,M_0)} \leq C_{r,s}\quad\mbox{and}\\
&\Norm{ |(M_0 -h)^{r}|\log(M_0 -h)|^{-s + 3}\partial_h^{r}\partial_{\theta}^s( v(\theta,h))}{L^{^\infty} ( ( -\pi,\pi) \times (0,M_0)} \leq C_{r,s}
\end{split}
\end{equation}
for some constants $C_{r,s}$, and 
$$\Norm{|\log(M_0 -h)|^{-s} \partial_\theta^s (x(\theta,h))}{L^{^\infty} ( ( -\pi,\pi) \times (0,M_0)} \leq C_s$$
for $s \geq 1$ and some constant $C_s$. 
\end{proposition}

\begin{proof}
Let us first prove \eqref{prop79-1} and \eqref{prop79-2}, starting with the study of $\omega_{\circ}$ when $h\to M_{0}^{-}.$
We have, by using \eqref{rachm1}, that
$$\omega_{\circ}(h) = \frac{\pi\sqrt{M_0}  }{2  \bK(k(h))} \sim \frac{\pi\sqrt{M_0}  }{ (- \log( 1 - k(h))) } $$
and we obtain the result using 
$$1 - k(h) = 1 - \sqrt{1 + \frac{h - M_0}{2M_0}} \sim \frac{M_0 - h}{4M_0}.  $$
This proves the first part of \eqref{prop79-1}. Note that the second estimate of \eqref{prop79-2} follows from the estimate on the function $1/\bK$ of \eqref{rachm}, and the smoothness of $k(h)$ in the vicinity $h\sim M_{0}.$\\
When $h \to -M_0,$ $\omega_{\circ}(h)$ is an analytic function of $k(h)^2 = \frac{h+M_0}{2M_0},$ and we have using \eqref{rachm0}
\begin{eqnarray*}
\omega_{\circ}(h) = \frac{\pi\sqrt{M_0}  }{2  \bK(k(h))}  &=& \sqrt{M_0} - \frac{\sqrt{M_0}}{4} k(h)^2 +  \mathcal{O}((h+M_0)^2) \\
&=& \sqrt{M_0}  - \frac{1}{8\sqrt{M_0}} (h + M_0) +  \mathcal{O}((h+M_0)^2). 
\end{eqnarray*}
The first estimate of \eqref{prop79-2} follows easily.

Let us now study the functions $x(\theta,h)$ and $v(\theta,h).$ Using \eqref{amuk} and the expression \eqref{220} and \eqref{221} of $x(\theta,h)$ and $v(\theta,h)$, we obtain the expansions
$$v(\theta,h) =   \sqrt{M_0} \frac{4\pi}{ \bK(k(h))}\sum_{m = 0}^{\infty} \frac{q(k(h))^{m + \frac{1}{2}}}{1 + q(k(h))^{2m + 1}} \cos\left( (2 m +1) \left(\theta +  \frac{\pi}{2}\right)\right) $$
and
\begin{multline*}
\sin(x(\theta,h)/2)=k(h) \sn\left(\frac{2}{\pi} \bK(k(h)) \left(\theta +  \frac{\pi}{2}\right) ,k(h)\right) \\
= \frac{2\pi}{ \bK(k(h))}\sum_{m = 0}^{\infty} \frac{q(k(h))^{m + \frac{1}{2}}}{1 - q(k(h))^{2m + 1}} \sin\left( (2 m +1) \left(\theta +  \frac{\pi}{2}\right)\right).
\end{multline*}
This, together with the fact that $\arcsin(z) \sim z$ is analytic in the vicinity of $z = 0$ and the expansions \eqref{rachm0}, shows that $v(\theta,h)$ and $x(\theta,h)$ are analytic functions of $\sqrt{h + M_0}$ when $h \to -M_0$, and that 
$$v(\theta,h) \sim 8 \sqrt{q(k(h))} 
\sqrt{M_0} \cos\left(\theta +  \frac{\pi}{2}\right) \sim - 2 k(h)\sqrt{M_0} \sin(\theta)$$
and
$$x(\theta,h) \sim 8 \sqrt{q(k(h))} \sin\left(\theta+  \frac{\pi}{2}\right) \sim  2 k(h) \cos(\theta)$$
which yields asymptotics \eqref{prop79-3}, and estimates \eqref{prop79-4} follow easily.\\
It remains to prove \eqref{prop79-5}. The analysis is similar to what we did for $U_{+}$ and $U_{-},$ as we have 
$$v(\theta,h) =  \sqrt{M_0} \frac{4\pi}{ \bK(k(h))} R(\theta,q(k(h)))$$
with 
$$R(\theta,q) = \sum_{m = 0}^{\infty} \frac{q^{m + \frac{1}{2}}}{1 + q^{2m + 1}} \cos\left( (2 m +1) \left(\theta +  \frac{\pi}{2}\right)\right).$$
By doing an analysis similar to the one performed for  $U_+$ and $U_-$, we have 
$$\Norm{\partial_q^r \partial_\theta^s R(\theta,q(k(h)))}{L^{\infty}(0,M_0)} \leq C \log\left( \frac{1}{ M_0 - h}\right) ^{r + s +1}, $$
Moreover \eqref{rachm} shows that
$$\partial_h^r q(k(h)) = \mathcal{O}\left(\left(\frac{1}{M_0 -h}\right)^{r} \frac{1}{\log (M_0 -h)^{2}}\right), $$
and 
$$\partial_h^r \frac{1}{\bK(k(h))} = \mathcal{O}\left(\left(\frac{1}{M_0 -h}\right)^{r} \frac{1}{\log (M_0 -h)^{2}}\right). $$
We deduce from these estimates that we have the same asymptotics as in the case of $U_+$ and $U_-$: 
$$\partial_h^r \partial_{\theta}^s v(\theta,h)  = \mathcal{O} \left( \left(\frac{1}{h-M_0}\right)^{r} \log \left(\frac{1}{h - M_0}\right)^{s-3}\right).$$
Now we can perform a similar analysis for
\begin{multline*}
\sin(x(\theta,h)/2)=k(h) \sn\left(\frac{2}{\pi} \bK(k(h)) \left(\theta +  \frac{\pi}{2}\right) ,k(h)\right) \\
= \frac{2\pi}{ \bK(k(h))}\sum_{m = 0}^{\infty} \frac{q(k(h))^{m + \frac{1}{2}}}{1 - q(k(h))^{2m + 1}} \sin\left( (2 m +1) \left(\theta +  \frac{\pi}{2}\right)\right)
\end{multline*}
after noticing that 
$$\frac{1}{1 - q^{2m+1}} = \frac{1}{1 - q} \left(\frac{1 - q}{1 - q^{2m+1}}\right) \leq \frac{C}{2 m +1}\left(\frac{1}{1 -q}\right)$$
when $q \in (\frac12,1)$. To obtain the conclusion for $x(\theta,h)$, we just have to be careful as $\arcsin$ has singularities in $\pm 1$:
recall that we have the expansion (see (4.4.41) in \cite{abramo}) 
$$\arcsin(x) = - \pi/2 + \sqrt{2(1+x)} \left(\sum_{n = 0}^{+\infty} \frac{(2n)!}{8^n (2n+1)(n!)^2} (1+x)^n \right)$$
and
$$\arcsin(x) = \pi/2 - \sqrt{2(1-x)} \left(\sum_{n = 0}^{+\infty} \frac{(2n)!}{8^n (2n+1) (n!)^2} (1-x)^n \right). $$
In our context, it will happen at the point $\pm x_0(h)$ and the singularity will be of order $\sqrt{M_0 - h}$. However, this singularity is weaker than the other one coming from functions $q(z)$ and $\bK(z)$ in the vicinity $z\sim 1$ (see \eqref{rachm1}). This finishes the proof of \eqref{prop79-5}.
\end{proof}
 The Fourier expansion of cosine and sine functions are given by the following result: 

\begin{proposition}
\label{prop710}
For $(\theta,a)\in  (-\pi,\pi)\times J_\circ,$ 
$$\cos(x(\theta,a)) = \sum_{\ell \in \Z} C_\ell^\circ(a) e^{i\ell \theta}
\quad \mbox{and} \quad 
\sin(x(\theta,a)) = \sum_{\ell \in \Z} S_\ell^\circ(a) e^{i\ell \theta}$$
with, in terms of the variable $h$, 
\begin{equation}
\label{coeffouriercos2}
\begin{split}
&C^\circ_0(h) = - 1 +  2\frac{\bE(k(h))}{\bK(k(h))} \\
&C^\circ_{2\ell}(h) = (-1)^{|\ell|} \frac{2\pi^2  }{ \bK(k(h))^2} \left( \frac{ |\ell| q(k(h))^{|\ell|}}{1 - q(k(h))^{2|\ell|}} \right), \quad \ell \neq 0\\
&C^\circ_{2\ell +1} = 0,
\end{split}
\end{equation}
and
\begin{equation}
\label{coeffouriersin2}
\begin{split}
&S^\circ_{2\ell}(h) = 0 \\
&S^\circ_{2\ell - 1}(h) = S^\circ_{-(2\ell - 1)}(h)=\frac{(-1)^{\ell -1}\pi^2}{2 \bK(k(h))^2}  \left(\frac{(2\ell-1) q(k(h))^{\ell-\frac{1}{2}}}{1 - q(k(h))^{2\ell-1}}\right), \quad \ell\geq 1. 
\end{split}
\end{equation}
\end{proposition}

\begin{proof}
Recall that we have
$$\sin(x(\theta,h)/2) =  k(h) \sn\left(\frac{2}{\pi} \bK(k(h)) \left(\theta+  \frac{\pi}{2}\right),k(h)\right).$$ 
By using the third formula of \eqref{classictrigo}, the definition of the function $\dn(u,k)$ and the expansion of $\sn(u,k)\dn(u,k)$ in \eqref{Milne}, we have
\begin{eqnarray*}
\sin(x(\theta,h)) &=&2\sin(x(\theta,h)/2)\sqrt{1-\sin^{2}(x(\theta,h)/2)}\\
&=& 2 k(h) \sn \left(\frac{2}{\pi} \bK(k(h)) \left(\theta+  \frac{\pi}{2}\right),k(h)\right) \dn \left(\frac{2}{\pi} \bK(k(h)) \left(\theta+  \frac{\pi}{2}\right),k(h)\right)\\
&=&   \frac{\pi^2}{ \bK(k(h))^2} \sum_{m = 1}^\infty \frac{(2 m - 1) q(k(h))^{m - \frac{1}{2}}}{1 + q(k(h))^{2m - 1}} \sin\left( (2 m - 1)  \left(\theta+  \frac{\pi}{2}\right)\right), 
%&=&\frac{\pi^2}{ \bK(k(h))^2} \sum_{m = 1}^\infty \frac{(2 m - 1) q(k(h))^{m - \frac{1}{2}}}{1 + q(k(h))^{2m - 1}} \left[\frac{e^{(2m-1)i\theta}e^{(2m-1)\frac{i\pi}{2}} - e^{-(2m-1)i\theta}e^{-(2m-1)\frac{i\pi}{2}}}{2i}\right],
\end{eqnarray*}
which yields \eqref{coeffouriersin2}.\\
Using now the first and fourth formulae of \eqref{classictrigo}, and the expansion of $\sn^{2}(u,k)$ from \eqref{Milne}, we obtain
\begin{eqnarray*}
\cos(x(\theta,h)) &=& \cos\left(2\arcsin\left(k(h) \sn\left(\frac{2}{\pi} \bK(k(h)) \left(\theta+  \frac{\pi}{2}\right),k(h)\right)\right)\right) \\
&=&2 \cos^{2}\left(\arcsin\left(k(h) \sn\left(\frac{2}{\pi} \bK(k(h)) \left(\theta+  \frac{\pi}{2}\right),k(h)\right)\right)\right) -1 \\
&=&1 - 2k(h)^2 \sn^2\left(\frac{2}{\pi} \bK(k(h)) \left(\theta+  \frac{\pi}{2}\right),k(h)\right) \\
&=& 1 - 2   \frac{\bK(k(h)) - \bE(k(h))}{ \bK(k(h))}  + \frac{4\pi^2 }{ \bK(k(h))^2} \sum_{m = 1}^\infty \frac{m q(k(h))^{m}}{1 - q(k(h))^{2m}} \cos( 2 m \theta +   m \pi),
\end{eqnarray*}
which yields \eqref{coeffouriercos2}.
\end{proof}

As in the case of $U_{\pm}$ we can establish estimates for general Fourier coefficients $f^{\circ}$ as follows:

\begin{proposition}
\label{prop711}
Assume that $f$ is a function satisfying 
$$
\max_{|\alpha |\leq m} \Norm{\partial_{x,v}^\alpha  f}{L^{\infty}(U_\circ)}\leq C_{m}
$$
for some $m\geq p+ 2,$ with $p$ defined by 
$$ p= \max \{n \geq 1  ,  \, \partial^\alpha_{x,v} f(0,0) = 0, \, \forall \alpha, \, 1 \leq |\alpha | \leq n\}$$
(and with the convention that $p=0$ if this set is empty).
Then, as long as $r+s\leq m$ and $s+p+2\leq m,$ we have that for $\ell \neq 0$,
\begin{equation}
\begin{split}
&\partial_h^r f_\ell^{\circ}(h) = \mathcal{O} \left( \frac{1}{|\ell|^s }\Big(\frac{1}{h-M_0}\Big)^{r} \log \Big(\frac{1}{h - M_0}\Big)^{s}\right) \quad\mbox{when}\quad h \to M_0^-,
\end{split}
\end{equation}
and that 
\begin{equation}
\label{gla1}
f_{\ell}^{\circ}(h) = { 1 \over |\ell|^s}\left( c_{\ell}\,(h+ M_{0})^{p+1 \over 2} + (h+ M_{0})^{p+2 \over 2}  r_{\ell}(\sqrt{h+ M_{0}})\right),
\end{equation}
where $c_{\ell}$ is a number uniformly bounded in $\ell$ and $r_{\ell} \in W^{m-(p+2)- s,\infty}$ uniformly in $\ell.$
\end{proposition}

\begin{proof} 
The estimates near the separatix are exactly the same as in the case of $U_+$ and $U_-$. Let us then focus on the asymptotic near $h = -M_0$. Taylor-expanding, we can always write that
$$f(x,v) = f(0,0) + F^{1}(x,v) \cdot (x,v)= f(0,0) + F^{1}(0,0) \cdot (x,v) + F^{2}(x,v) \cdot (x,v)^{(2)}$$
where $F^{1}(x,v)$ is linear and $F^{2}(x,v)$ is bilinear. We may write for $\ell \neq 0$ that
\begin{equation}
\label{gla1p1}
f_{\ell}^{\circ}(a)= {1 \over 2 \pi} \int_{(-\pi,\pi)} \left[ F^{1}(0,0) \cdot (x,v) + F^{2}(x(\theta, a), v(\theta, a) ) \cdot   (x(\theta, a), v(\theta, a) )^{(2)}\right] e^{-i \ell \theta} \, \dd \theta.
\end{equation}
We obtain then from Proposition \ref{prop79} that $U(\theta, a)= (x(\theta, a), v(\theta, a))$ can be expanded  when $h$ is near $-M_{0}$ as 
$$ U(\theta, a) = \sum_{n \geq 1} a_{n}(\theta) (h+ M_{0})^{n \over 2},$$
where the functions $a_{n}$ are smooth since $(x(\theta, a), v(\theta, a))$ is an analytic function of $\sqrt{h+ M_{0}}$ uniformly in $\theta.$ By plugging this expansion in \eqref{gla1p1} and by  integrating by parts $s$ times we get \eqref{gla1}  for $p=0.$ If  $p>0$, we just notice that by further Taylor expansion, we have
$$f(x,v) = f(0,0) + F^{p+1}(0,0) \cdot (x,v)^{(p)} + F^{p+2}(x,v) (x,v)^{(p+1)},$$ 
where $F^{p+1}(x,v)$ is $p+1$-linear and $F^{p+2}$ $p+2$ linear. It suffices then to proceed as above.
\end{proof}

From this result, we obtain the following: 

\begin{proposition}
\label{decaydansloeil}
Assume that $f$ an $\varphi$ are real functions satisfying 
$$\max_{|\alpha |  \leq m} \Norm{\partial_{x,v}^\alpha  f}{L^{\infty}(U_\circ)}\leq C_{m}\quad\mbox{and}\quad\max_{|\alpha |  \leq M} \Norm{\partial_{x,v}^\alpha  \varphi}{L^{\infty}(U_\circ)}\leq C_{r,s}$$
for some $m\geq 0.$ Let $p$ and $q$ defined as
\begin{eqnarray*} &&p= \max \{n \geq 1  ,  \, \partial^\alpha_{x,v} f(0,0) = 0, \, \forall \alpha, \, 1 \leq |\alpha | \leq n\}, \\
&&q= \max \{n \geq 1  ,  \, \partial^\alpha_{x,v} \varphi(0,0) = 0, \, \forall \alpha, \, 1 \leq |\alpha | \leq n\}.
\end{eqnarray*}
Then, for $m \geq  5 + p  + { p+ q \over 2} $, $M\geq 7 + q +  { p+ q \over 2 }$, $M \geq m+2$,  we have for $t \geq 1$
$$\left|\int_{U_\circ} f(x,v) \varphi(\psi_{t}(x,v)) \dd v -  \int_{-M_0}^{M_0} f_0^{\circ}(h) \varphi_0^{\circ}(h) \frac{1}{\omega_{\circ}(h)} \dd h \right|  \leq \frac{C}{(1 + t)^{  \frac{p+q}{2}  +2 } }.$$
\end{proposition}

\begin{proof}
We begin as in Proposition \ref{decaydehors} and write that
$$\int_{U_\circ} f(x,v) \varphi(\psi_t(x,v)) \dd x\dd v  =\sum_{\ell\in \Z} \int_{-M_0}^{M_0}f_\ell^{\circ}(h) \varphi_{-\ell}^{\circ}(h)  e^{i t \ell \omega_{\circ}(h) } \frac{1}{\omega_{\circ}(h)} \dd h.$$ 
By taking a smooth nonnegative  function $\chi (s) $ such that $\chi= 1$ for $s \leq {\delta}$ and $\chi= 0$
for $s \geq 2 \delta $ with $\delta$ small enough, we can split the integral into
\begin{multline*} 
\int_{-M_0}^{M_0}f_\ell^{\circ}(h) \varphi_{-\ell}^{\circ}(h)  e^{i t \ell \omega_{\circ}(h) } \frac{1}{\omega_{\circ}(h)} d h 
 = \int_{-M_0}^{M_0} \chi \left( { h + M_{0} \over M_{0}} \right)f_\ell^{\circ}(h) \varphi_{-\ell}^{\circ}(h)  e^{i t \ell \omega_{\circ}(h) } \frac{1}{\omega_{\circ}(h)} \dd h  
\\  +   \int_{-M_0}^{M_0}  \left(1 -  \chi \left( { h + M_{0} \over M_{0}} \right)\right)f_\ell^{\circ}(h) \varphi_{-\ell}^{\circ}(h)  e^{i t \ell \omega_{\circ}(h) } \frac{1}{\omega_{\circ}(h)} \dd h  = I_{\ell}^1 + I_{\ell}^2.
\end{multline*}  
As in the proof of Proposition \ref{decaydehors}, the idea is again to integrate by parts as long as we can, {\it i.e.} as long as the contributions from the boundary points $h\sim \pm M_{0}$ vanish. The term $I_{\ell}^2$ can be handled as before for $U_+$ and $U_-:$ as $\partial_{h}\omega_{\circ}$ does not vanish, only the contribution at the separatix $h \sim M_0$ matters, and this yields a decay by $(1 + t)^{-r}$ assuming enough regularity. As a matter of fact, we just need to take $m$ large enough in order to choose $r \geq {p+q \over 2} + 2.$\\
We shall now focus on $I_{\ell}^1$ which contains the contribution from the center $h\sim -M_{0}.$ By using Proposition \ref{prop711}, we can expand $I_{\ell}^1$ under the form 
\begin{multline*}
I_{\ell}^1= {1 \over |\ell|^s} \tilde{c}_{\ell} \int_{-M_{0}}^{M_{0}} ( h+ M_{0})^{ 1 + { p+q \over 2 }}   \chi \left( { h + M_{0} \over M_{0}} \right)e^{i t \ell \omega_{\circ}(h) } \frac{1}{\omega_{\circ}(h)} \dd h
\\ +  {1 \over |\ell|^s} \tilde{c}_{\ell} \int_{-M_{0}}^{M_{0}} ( h+ M_{0})^{ 1 + { 1+ p+q \over 2 }}  \chi \left( { h + M_{0} \over M_{0}} \right)     \tilde{r}_{\ell}(\sqrt{h+ M_{0}}) \frac{1}{\omega_{\circ}(h)}  e^{i t \ell \omega_{\circ}(h)} \dd h,
\end{multline*}
where $\tilde{r}_{\ell}$ is uniformly in $\ell$  in $W^{m- p - 2}$ (since we always assume that $M$ is much bigger than $m$, 
$ M \geq m+s$), and where the constant $c_{\ell}$ is uniformly bounded in $\ell.$ It is important to notice that $\omega_{\circ}(h)$ and all its derivatives are non-zero smooth functions in $]-M_{0},c]$ for any $c\in (0,M_{0}),$ so that during the integration by parts, $\omega_{\circ}$ will not play any major part.\\
Let us first consider the case where $p+q$ is even, and write $p+q= 2k.$ Then the polynomial contributions in $(h+M_{0})$ are $ (h+M_{0})^{k+1}$ in the first integral and $(h+M_{0})^{k+1+{1 \over 2}}$ in the second integral above. We can thus integrate by parts $k+2$ times in each of the two integrals
(as in the proof of Proposition \ref{decaydehors}), in order to obtain that 
$$| I_{\ell}^1| \lesssim {1 \over |\ell|^s} { 1 \over (1 + t)^{k+2}}.$$
Taking $s=2$ and summing with respect to $\ell,$ we get the result.\\
In the case $ p+q= 2k+1,$ the polynomial contributions in $(h+M_{0})$ are $(h+M_{0})^{k+{3 \over 2}}$ in the first integral and  $(h+M_{0})^{k+2}$ in the second integral.  For the latter we can thus integrate by parts $k+3$ times as previously to get a decay like ${ 1 / (1+ t)^{k+3}},$ which is $ { 1 / (1+ t)^{ { p+ q \over 2} + { 5 \over 2}}},$ and is already faster than the expected decay. For the first integral, we can integrate by parts $k+3$ times, except for the most singular term where we can integrate by parts only  $k+2$ times without boundary terms to obtain integrals under the form
$$ \tilde{I}_{\ell}^1= { 1 \over  (1 + t)^{k+2} |\ell|^s} \tilde {c}_{\ell} \int_{-M_{0}}^{M_{0}} ( h+ M_{0})^{ - { 1\over 2 }}   \chi \left( { h + M_{0} \over M_{0}} \right) \tilde{\chi}(h+ M_{0})e^{i t \ell \omega_{\circ}(h) } \dd h.$$
where $\tilde \chi$ is a smooth function. Next, since $\partial_{h}\omega_{\circ}$ does not vanish, we can make the change of variable $u = \omega_{\circ} (h) - \omega_{\circ} (-M_{0})$. By observing that this allows to write  $ h+ M_{0} = u  A(u)$ where $A$ is smooth and does not vanish
(so in particular, we have that  $ A(0) \neq 0$), we can thus write
$$ \tilde{I}_{\ell}^1= { 1 \over (1 + t)^{k+2} |\ell|^s} \tilde {c}_{\ell} \int_{0}^X  { 1 \over u^{1 \over 2}} \Psi (u) e^{it \ell u }\, \dd u$$
where $\Psi$ is smooth and compactly supported in $[0, X)$. Taylor-expanding the function $\Psi$, we obtain that
$$  \tilde{I}_{\ell}^1= { 1 \over  (1 + t)^{k+2} |\ell|^s} \tilde {c}_{\ell} \int_{0}^X  { 1 \over u^{1 \over 2}}  e^{it \ell u }\, \dd u
+ { 1 \over  (1 + t)^{k+2} |\ell|^s} \tilde {c}_{\ell} \int_{0}^X  u^{1 \over 2} \Psi^1 (u) e^{it \ell u }\, \dd u.$$
For the second integral above, we can integrate by parts once to obtain an estimate by ${ 1 \over  (1 + t)^{k+3} |\ell|^s}.$ To handle the first integral we use Lemma \ref{integrale} below, which yields the decay
$$1/ (1 + t)^{k+2+ {1 \over 2}}.$$
By noticing that $ k+ {1 \over 2} + 2= {p+ q \over 2} + 2$, we finally get the result.  
\end{proof}

\begin{lemma}
\label{integrale}
Consider the integral
$$ I(t) = \int_{0}^X { 1 \over u^{1\over 2} } e^{it u} \, \dd u.$$
Then we have that for $t\geq 1$, 
$$ |I(t)| \lesssim {1 \over t^{1 \over 2} }.$$
\end{lemma}

\begin{proof}
Let us set $v=tu$ in the integral, we obtain that
$$ I(t) = { 1 \over t^{1 \over 2} } \int_{0}^{ tX} { 1 \over v^{1 \over 2}} e^{i v}\, \dd v
=   { 1 \over t^{1 \over 2} } \int_{0}^{ 1} { 1 \over v^{1 \over 2}} e^{i v}\, \dd v
+  { 1 \over t^{1 \over 2} } \int_{1}^{ tX} { 1 \over v^{1 \over 2}} e^{i v}\, \dd v$$
(assuming that $t$ is sufficiently large so that $tX \geq 1$). The first integral in the right-hand side above is clearly uniformly bounded by $1/ t^{1 \over 2}$. 
For the  second integral in the above right hand side, we can integrate by parts once to get that
$$   \left|{ 1 \over t^{1 \over 2} } \int_{1}^{ tX} { 1 \over v^{1 \over 2}} e^{i v}\, \dd v \right| 
\lesssim {1 \over t^{1 \over 2}} \left( 1 + \int_{1}^{+ \infty} { 1 \over v^{3 \over 2}}\, \dd v \right).$$
\end{proof}

\end{document}